\documentclass[11pt, oneside, reqno]{amsart}
\usepackage[utf8]{inputenc}
\usepackage{mathrsfs}
\usepackage{bm}
\usepackage{amstext}
\usepackage{amsthm}
\usepackage{amssymb}
\usepackage{mathtools}
\usepackage[all]{xy}
\usepackage{cleveref}
\usepackage{graphics}
\numberwithin{equation}{section}
\numberwithin{figure}{section}
\theoremstyle{plain}
\newtheorem{thm}{Theorem}[section]
  \crefname{thm}{Theorem}{Theorems}
  \newtheorem{lem}[thm]{Lemma}
  \crefname{lem}{Lemma}{Lemmas}
  \newtheorem{prop}[thm]{Proposition}
  \crefname{prop}{Proposition}{Propositions}
  
	\crefname{cor}{Corollary}{Corollaries}
  \newtheorem*{theor1}{Theorem 1}
  
\theoremstyle{definition}
  \newtheorem{defi}[thm]{Definition}
  \crefname{defi}{Definition}{Definitions}
  \theoremstyle{remark}
  
  \crefname{ntn}{Notation}{Notations}
	 \theoremstyle{remark}
  \newtheorem{rem}[thm]{Remark}
  \crefname{rem}{Remark}{Remarks}
  \newtheorem{ex}[thm]{Example}
  \crefname{ex}{Example}{Examples}
  
\usepackage{a4wide}

\def\r{\mathbb{R}}
\def\c{\mathbb{C}}

\def\z{\mathbb{Z}}

\newcommand{\Spec}{\operatorname{Spec}}

\makeatother

\makeatletter
\ifcsname phantomsection\endcsname
    \newcommand*{\qrr@gobblenexttocentry}[5]{}
\else
    \newcommand*{\qrr@gobblenexttocentry}[4]{}
\fi
\newcommand*{\addsubsection}{%
    \addtocontents{toc}{\protect\qrr@gobblenexttocentry}%
    \subsection}
\makeatother
 \makeatletter
    
    \@addtoreset{equation}{section}
  \makeatother
\begin{document}
\title[Semi-toric degenerations arising from cluster structures]{Semi-toric degenerations of Richardson varieties arising from cluster structures on flag varieties}

\author{Naoki FUJITA}

\address[Naoki FUJITA]{Graduate School of Mathematical Sciences, The University of Tokyo, 3-8-1 Komaba, Meguro-ku, Tokyo 153-8914, Japan.}

\email{nfujita@ms.u-tokyo.ac.jp}



\thanks{The work of the author was supported by Grant-in-Aid for JSPS Fellows (No.\ 19J00123) and by Grant-in-Aid for Early-Career Scientists (No.\ 20K14281).}

\date{}

\begin{abstract}
A toric degeneration of an irreducible variety is a flat degeneration to an irreducible toric variety. 
In the case of a flag variety, its toric degeneration with desirable properties induces degenerations of Richardson varieties to unions of irreducible closed toric subvarieties, called semi-toric degenerations.
For instance, Morier-Genoud proved that Caldero's toric degenerations arising from string polytopes have this property. 
Semi-toric degenerations are closely related to Schubert calculus. 
Indeed, Kogan--Miller constructed semi-toric degenerations of Schubert varieties from Knutson--Miller's semi-toric degenerations of matrix Schubert varieties which give a geometric proof of the pipe dream formula of Schubert polynomials. 
In this paper, we focus on a toric degeneration of a flag variety arising from a cluster structure, and prove that it induces semi-toric degenerations of Richardson varieties. 
Our semi-toric degeneration can be regarded as a generalization of Morier-Genoud's and Kogan--Miller's semi-toric degenerations. 
\end{abstract}
\maketitle
\tableofcontents 
\section{Introduction}\label{s:section}
A toric degeneration of an irreducible variety is a flat degeneration to an irreducible toric variety. 
In the case of a flag variety, its toric degeneration has been studied from various points of view such as standard monomial theory \cite{Chi, GL}, string parametrizations of dual canonical bases \cite{AB, Cal}, PBW (Poincar\'e--Birkhoff--Witt) filtrations \cite{FeFL1, FeFL2, FeFL3}, and so on; see \cite{FaFL} for a survey on this topic. 
A toric degeneration of a flag variety with desirable properties induces degenerations of Richardson varieties to unions of irreducible closed toric subvarieties, called semi-toric degenerations.
For instance, Morier-Genoud \cite{Mor} proved that Caldero's toric degenerations \cite{Cal} arising from string polytopes have this property. 
Semi-toric degenerations are closely related to Schubert calculus. 
Indeed, Kogan--Miller \cite{KoM} constructed semi-toric degenerations of Schubert varieties from Knutson--Miller's semi-toric degenerations \cite{KnM} of matrix Schubert varieties which give a geometric proof of the pipe dream formula of Schubert polynomials. 
The author \cite{Fuj2} showed that Morier-Genoud's semi-toric degeneration can be regarded as a generalization of Kogan--Miller's semi-toric degeneration.

A systematic method of constructing toric degenerations is given by the theory of Newton--Okounkov bodies (see \cite{And, HK}). 
A Newton--Okounkov body $\Delta(Z, \mathcal{L}, v)$ is a convex body constructed from a projective variety $Z$ with a line bundle $\mathcal{L}$ on $Z$ generated by global sections and with a higher rank valuation $v$ on the function field $\mathbb{C}(Z)$, which was introduced by Okounkov \cite{Oko1, Oko2, Oko3}, and afterward developed independently by Lazarsfeld--Mustata \cite{LM} and by Kaveh--Khovanskii \cite{KK1}. 
In the case of flag varieties, their Newton--Okounkov bodies realize the following representation-theoretic polytopes:
\begin{enumerate}
\item[{\rm (i)}] Berenstein--Littelmann--Zelevinsky's string polytopes \cite{Kav},
\item[{\rm (ii)}] Nakashima--Zelevinsky polytopes \cite{FN},
\item[{\rm (iii)}] Feigin--Fourier--Littelmann--Vinberg polytopes \cite{FeFL3, Kir},
\end{enumerate}
which all induce toric degenerations. 
In the present paper, we focus on Newton--Okounkov bodies of flag varieties arising from cluster structures. 

The theory of cluster algebras was originally introduced by Fomin--Zelevinsky \cite{FZ:ClusterI, FZ:ClusterIV} to develop a combinatorial approach to total positivity in reductive groups and to Lusztig's dual canonical basis. 
Fock--Goncharov \cite{FG} introduced a pair $(\mathcal{A}, \mathcal{X})$ of cluster varieties, called a cluster ensemble, which gives a more geometric point of view to the theory of cluster algebras. 
Gross--Hacking--Keel--Kontsevich \cite{GHKK} developed the theory of cluster ensembles using methods in mirror symmetry, and showed that the theory of cluster algebras can be used to construct toric degenerations of projective varieties.
The author and Oya \cite{FO2} gave another approach to toric degenerations by relating the theory of cluster algebras with Newton--Okounkov bodies, and constructed a family of toric degenerations of flag varieties using cluster structures. 
Our aim of the present paper is to prove that the toric degenerations in this family induce semi-toric degenerations of Richardson varieties.

To state our result more explicitly, let $G$ be a simply-connected semisimple algebraic group over $\mathbb{C}$, $B \subseteq G$ a Borel subgroup, $P_+$ the set of dominant integral weights, and $P_{++}$ the set of regular dominant integral weights. 
For $\lambda \in P_+$, denote by $\mathcal{L}_\lambda$ the corresponding line bundle on $G/B$ which is generated by global sections. 
Let $W$ be the Weyl group, $w_0 \in W$ the longest element, and $\leq$ the Bruhat order on $W$.
For $w \in W$, we denote by $U_w ^- \subseteq G$ the associated unipotent cell.
Berenstein--Fomin--Zelevinsky \cite{BFZ} proved that the coordinate ring $\c[U_w ^-]$ admits an upper cluster algebra structure. 
Note that the unipotent cell $U_{w_0} ^-$ associated with $w_0$ is naturally thought of as an open subvariety of the full flag variety $G/B$. 
Hence we have $\c(G/B) \simeq \c(U_{w_0} ^-)$, and the upper cluster algebra $\c[U_{w_0} ^-]$ is identified with a $\c$-subalgebra of the function field $\c(G/B)$.
Using this identification, the author and Oya \cite{FO2} constructed for $\lambda \in P_+$ a family $\{\Delta(G/B, \mathcal{L}_\lambda, v_{\mathbf s})\}_{{\mathbf s} \in \mathcal{S}}$ of Newton--Okounkov bodies parametrized by the set $\mathcal{S}$ of Fomin--Zelevinsky seeds for $\c[U_{w_0} ^-]$ such that 
\begin{itemize}
\item this family contains string polytopes and Nakashima--Zelevinsky polytopes up to unimodular transformations;
\item $\Delta(G/B, \mathcal{L}_\lambda, v_{\mathbf s})$, ${\mathbf s} \in \mathcal{S}$, are all rational convex polytopes;
\item $\Delta(G/B, \mathcal{L}_\lambda, v_{\mathbf s})$, ${\mathbf s} \in \mathcal{S}$, are all related by tropicalized cluster mutations;
\item for $\lambda \in P_{++}$ and ${\mathbf s} \in \mathcal{S}$, there exists a toric degeneration of $G/B$ to the irreducible normal projective toric variety $X(\Delta(G/B, \mathcal{L}_\lambda, v_{\mathbf s}))$ corresponding to $\Delta(G/B, \mathcal{L}_\lambda, v_{\mathbf s})$.
\end{itemize}
See Theorems \ref{t:NO_body_parametrized_by_seeds}, \ref{t:NO_body_parametrized_by_seeds_string}, \ref{t:NO_body_parametrized_by_seeds_NZ} for more details. 
For $v, w \in W$ such that $v \leq w$, we denote by $X_w^v \subseteq G/B$ the corresponding Richardson variety.
The following is the main result of the present paper.

\begin{theor1}[{see \cref{t:main_result_semi-toric}}]
Let $\lambda \in P_{++}$, ${\mathbf s} \in \mathcal{S}$, and $v, w \in W$ such that $v \leq w$.
Then, under the toric degeneration of $G/B$ to $X(\Delta(G/B, \mathcal{L}_\lambda, v_{\mathbf s}))$, the Richardson variety $X_w^v$ degenerates to a union of irreducible closed toric subvarieties of $X(\Delta(G/B, \mathcal{L}_\lambda, v_{\mathbf s}))$.
\end{theor1}

The semi-toric degenerations in Theorem 1 can be regarded as a generalization of Morier-Genoud's and Kogan--Miller's semi-toric degenerations. 

This paper is organized as follows. 
In Section \ref{s:NO_bodies}, we review the definition of Newton--Okounkov bodies.
The construction of toric degenerations given in \cite{And, HK} is also explained.
In Section \ref{s:crystal_bases}, we recall some basic facts on crystal bases, string polytopes, and Nakashima--Zelevinsky polytopes. 
We also review results by Morier-Genoud \cite{Mor} in this section.
In Section \ref{s:Cluster_valuation}, we recall some basic definitions on cluster algebras. 
We also see a construction of valuations using cluster structures, following \cite{FO2}.  
In Section \ref{s:unipotent_cell}, we review cluster structures on unipotent cells, and recall results in \cite{FO2}.  
In Section \ref{s:main_result}, we prove Theorem 1 above.

\section{Newton--Okounkov bodies}\label{s:NO_bodies}

In this section, we review some basic definitions and facts on Newton--Okounkov bodies, following \cite{And, HK, Kav, KK1, KK2}. 
Let $R$ be a $\mathbb{C}$-algebra without nonzero zero-divisors. 
We take $m \in \mathbb{Z}_{>0}$ and fix a total order $\leq$ on $\mathbb{Z}^m$ respecting the addition. 

\begin{defi}
A map $v \colon R \setminus \{0\} \rightarrow \mathbb{Z}^m$ is called a \emph{valuation} on $R$ if the following conditions hold: 
for $\sigma, \tau \in R \setminus \{0\}$ and $c \in \mathbb{C}^\times \coloneqq \mathbb{C} \setminus \{0\}$,
\begin{enumerate}
\item[{\rm (i)}] $v(\sigma \cdot \tau) = v(\sigma) + v(\tau)$,
\item[{\rm (ii)}] $v(c \cdot \sigma) = v(\sigma)$, 
\item[{\rm (iii)}] $v (\sigma + \tau) \geq \min\{v(\sigma), v(\tau)\}$ unless $\sigma + \tau = 0$. 
\end{enumerate}
\end{defi}

The following is a fundamental property of valuations.

\begin{prop}[{see, for instance, \cite[Proposition 1.8]{Kav}}]\label{p:property_valuation}
Let $v$ be a valuation on $R$, and take $\sigma_1, \ldots, \sigma_k \in R \setminus \{0\}$ such that $v(\sigma_1), \ldots, v(\sigma_k)$ are all distinct. 
Then, for $c_1, \ldots, c_k \in \mathbb{C}$ such that $\sigma \coloneqq c_1 \sigma_1 + \cdots + c_k \sigma_k \neq 0$, it holds that
\[v(\sigma) = \min\{v(\sigma_\ell) \mid 1 \leq \ell \leq k,\ c_\ell \neq 0 \}.\]
\end{prop} 

For ${\bm a} \in \z^m$ and a valuation $v \colon R \setminus \{0\} \rightarrow \mathbb{Z}^m$, set 
\[R_{\bm a} \coloneqq \{\sigma \in R \setminus \{0\} \mid v(\sigma) \geq {\bm a}\} \cup \{0\},\]
which is a $\c$-subspace of $R$. 
Then the \emph{leaf} above ${\bm a} \in \z^m$ is defined as the quotient space $R[{\bm a}] \coloneqq R_{\bm a}/\bigcup_{{\bm a} < {\bm b}} R_{\bm b}$. 
We say that a valuation $v$ has \emph{1-dimensional leaves} if $\dim_\c(R[{\bm a}]) = 0\ {\rm or}\ 1$ for all ${\bm a} \in \mathbb{Z}^m$. 

\begin{ex}\label{ex:lowest_term_valuation}
Let $\mathbb{C}(t_1, \ldots, t_m)$ be the field of rational functions in $m$ variables. 
The fixed total order $\leq$ on $\mathbb{Z}^m$ induces a total order (denoted by the same symbol $\leq$) on the set of Laurent monomials in $t_1, \ldots, t_m$ as follows: 
\begin{center}
$t_1 ^{a_1} \cdots t_m ^{a_m} \leq t_1 ^{a_1 ^\prime} \cdots t_m ^{a_m ^\prime}$ if and only if $(a_1, \ldots, a_m) \leq (a_1 ^\prime, \ldots, a_m ^\prime)$. 
\end{center}
We define a map $v^{\rm low}_{\leq} \colon \mathbb{C}(t_1, \ldots, t_m) \setminus \{0\} \rightarrow \mathbb{Z}^m$ by
\begin{itemize}
\item $v^{\rm low} _{\leq} (f) \coloneqq (a_1, \ldots, a_m)$ for
\[f = c t_1 ^{a_1} \cdots t_m ^{a_m} + ({\rm higher\ terms}) \in \mathbb{C}[t_1, \ldots, t_m] \setminus \{0\},\]
where $c \in \mathbb{C}^\times$, and the summand ``(higher terms)'' stands for a linear combination of monomials bigger than $t_1 ^{a_1} \cdots t_m ^{a_m}$ with respect to $\leq$;
\item $v^{\rm low}_{\leq} (f/g) \coloneqq v^{\rm low} _{\leq} (f) - v^{\rm low} _{\leq} (g)$ for $f, g \in \mathbb{C}[t_1, \ldots, t_m] \setminus \{0\}$.
\end{itemize}
Then it follows that $v^{\rm low} _{\leq}$ is a valuation with respect to the total order $\leq$, which has $1$-dimensional leaves.
It is called the \emph{lowest term valuation} with respect to $\leq$.
\end{ex}

\begin{defi}[{see \cite[Section 1.2]{Kav} and \cite[Definition 1.10]{KK2}}]\label{d:Newton--Okounkov body}
Let $Z$ be an irreducible normal projective variety over $\mathbb{C}$ of complex dimension $m$, and $\mathcal{L}$ a line bundle on $Z$ generated by global sections. 
We fix a valuation $v \colon \mathbb{C}(Z) \setminus \{0\} \rightarrow \mathbb{Z}^m$ with $1$-dimensional leaves and a nonzero section $\tau \in H^0 (Z, \mathcal{L})$. 
Define a subset $S(Z, \mathcal{L}, v, \tau) \subseteq \mathbb{Z}_{>0} \times \mathbb{Z}^m$ by 
\[S(Z, \mathcal{L}, v, \tau) \coloneqq \bigcup_{k \in \z_{>0}} \{(k, v(\sigma/\tau^k)) \mid \sigma \in H^0(Z, \mathcal{L}^{\otimes k}) \setminus \{0\}\},\] 
and denote by $C(Z, \mathcal{L}, v, \tau) \subseteq \mathbb{R}_{\geq 0} \times \mathbb{R}^m$ the smallest real closed cone containing $S(Z, \mathcal{L}, v, \tau)$. 
Then we define a subset $\Delta(Z, \mathcal{L}, v, \tau) \subseteq \mathbb{R}^m$ by 
\[\Delta(Z, \mathcal{L}, v, \tau) \coloneqq \{{\bm a} \in \r^m \mid (1, {\bm a}) \in C(Z, \mathcal{L}, v, \tau)\},\] 
which is called the \emph{Newton--Okounkov body} of $(Z, \mathcal{L})$ associated with $v$ and $\tau$. 
\end{defi}

By the definition of valuations, we see that $S(Z, \mathcal{L}, v, \tau)$ is a semigroup. 
Hence the closed cone $C(Z, \mathcal{L}, v, \tau)$ and the set $\Delta(Z, \mathcal{L}, v, \tau)$ are both convex.
In addition, it follows by \cite[Theorem 2.30]{KK2} that $\Delta(Z, \mathcal{L}, v, \tau)$ is a convex body, i.e., a compact convex set. 
If $\mathcal{L}$ is very ample, then we see by \cite[Corollary 3.2]{KK2} that the real dimension of $\Delta(Z, \mathcal{L}, v, \tau)$ coincides with $m$.
If the semigroup $S(Z, \mathcal{L}, v, \tau)$ is finitely generated, then the Newton--Okounkov body $\Delta(Z, \mathcal{L}, v, \tau)$ is a rational convex polytope.
If the semigroup $S(Z, \mathcal{L}, v, \tau)$ is saturated in addition, that is, if for all $k \in \mathbb{Z}_{> 0}$ and ${\bm a} \in \mathbb{Z}^m$, $k {\bm a} \in S(Z, \mathcal{L}, v, \tau)$ implies ${\bm a} \in S(Z, \mathcal{L}, v, \tau)$, then ${\rm Proj} (\c[S(Z, \mathcal{L}, v, \tau)])$ is normal and isomorphic to the irreducible normal projective toric variety $X(\Delta(Z, \mathcal{L}, v, \tau))$ corresponding to the rational convex polytope $\Delta(Z, \mathcal{L}, v, \tau)$, where we take ``${\rm Proj}$'' for the natural $\z_{\geq 0}$-grading of $\c[S(Z, \mathcal{L}, v, \tau)]$ induced from the $\z_{>0}$-grading of $S(Z, \mathcal{L}, v, \tau)$.

\begin{rem}
Let $\tau^\prime \in H^0 (Z, \mathcal{L})$ be another nonzero section. 
Then it holds that $\Delta(Z, \mathcal{L}, v, \tau^\prime) = \Delta(Z, \mathcal{L}, v, \tau) + v(\tau/\tau^\prime)$. 
Thus, the Newton--Okounkov body $\Delta(Z, \mathcal{L}, v, \tau)$ is independent of the choice of a nonzero section $\tau \in H^0 (Z, \mathcal{L})$ up to translations by integer vectors. 
Hence it is also denoted simply by $\Delta(Z, \mathcal{L}, v)$.
\end{rem} 

We say that $Z$ admits a \emph{flat degeneration} to a variety $Z_0$ if there exists a flat morphism 
\[\pi \colon \mathfrak{X} \rightarrow \Spec(\c[z])\] 
of schemes such that the scheme-theoretic fiber $\pi^{-1}(z)$ (resp., $\pi^{-1}(0)$) over a closed point $z \in \c^\times$ (resp., the origin $0 \in \c$) is isomorphic to $Z$ (resp., $Z_0$). 
If $Z_0$ is a toric variety, then this degeneration is called a \emph{toric degeneration}.

\begin{thm}[{see \cite[Theorem 1]{And} and \cite[Corollary 3.14]{HK}}]\label{t:toric_deg}
Assume that $\mathcal{L}$ is very ample. 
If the semigroup $S(Z, \mathcal{L}, v, \tau)$ is finitely generated and saturated, then there exists a toric degeneration of $Z$ to the irreducible normal projective toric variety $X(\Delta(Z, \mathcal{L}, v, \tau))$ corresponding to the rational convex polytope $\Delta(Z, \mathcal{L}, v, \tau)$.
\end{thm}

More strongly, a concrete construction of the toric degeneration is given in the papers \cite{And, HK}.
Roughly speaking, the corresponding flat morphism $\pi \colon \mathfrak{X} \rightarrow \Spec(\c[z])$ is obtained by applying the \emph{Rees algebra construction} to a $\z_{\geq 0} \times \z^m$-filtration, given by the valuation $v$, on the $\z_{\geq 0}$-graded ring 
\[R(Z, \mathcal{L}) \coloneqq \bigoplus_{k \in \z_{\geq 0}} H^0(Z, \mathcal{L}^{\otimes k}).\]
Although the Rees algebra construction is a method for a $\z_{\geq 0}$-filtered algebra, we can reduce the problem to a $\z_{\geq 0}$-filtration on $R(Z, \mathcal{L})$; see the proofs of \cite[Proposition 3]{And} and \cite[Theorem 3.13]{HK} for more details.
Throughout the present paper, we always assume that a toric degeneration for a Newton--Okounkov body is constructed in this way.

\section{Crystal bases and their polyhedral parametrizations}\label{s:crystal_bases}

\subsection{Basic definitions on crystal bases}

In this subsection, we review some basic definitions and facts on crystal bases, following \cite{Kas2, Kas3, Kas4}. 
Let $G$ be a connected, simply-connected semisimple algebraic group over $\mathbb{C}$, $\mathfrak{g}$ its Lie algebra, and $H \subseteq G$ a maximal torus.
We denote by $\mathfrak{h} \subseteq \mathfrak{g}$ the Lie algebra of $H$, by $\mathfrak{h}^\ast \coloneqq {\rm Hom}_\mathbb{C} (\mathfrak{h}, \mathbb{C})$ its dual space, and by $\langle \cdot, \cdot \rangle \colon \mathfrak{h}^\ast \times \mathfrak{h} \rightarrow \mathbb{C}$ the canonical pairing. 
Let $P \subseteq \mathfrak{h}^\ast$ be the weight lattice, $I = \{1, 2, \ldots, n\}$ an index set for the vertices of the Dynkin diagram, $\{\alpha_i \mid i \in I\} \subseteq P$ the set of simple roots, and $\{h_i \mid i \in I\} \subseteq \mathfrak{h}$ the set of simple coroots.

\begin{defi}[{\cite[Definition 1.2.1]{Kas4}}]
A \emph{crystal} is a set $\mathcal{B}$ equipped with maps 
\begin{itemize}
\item ${\rm wt} \colon \mathcal{B} \rightarrow P$, 
\item $\varepsilon_i \colon \mathcal{B} \rightarrow \mathbb{Z} \cup \{-\infty\}$ and $\varphi_i \colon \mathcal{B} \rightarrow \mathbb{Z} \cup \{-\infty\}$ for $i \in I$,
\item $\tilde{e}_i \colon \mathcal{B} \rightarrow \mathcal{B} \cup \{0\}$ and $\tilde{f}_i \colon \mathcal{B} \rightarrow \mathcal{B} \cup \{0\}$ for $i \in I$,
\end{itemize}
where $-\infty$ (resp., $0$) is an additional element which is not included in $\mathbb{Z}$ (resp., $\mathcal{B}$), such that the following conditions hold: for $i \in I$ and $b, b^\prime \in \mathcal{B}$,
\begin{enumerate}
\item[(i)] $\varphi_i(b) = \varepsilon_i(b) + \langle{\rm wt}(b), h_i \rangle$,
\item[(ii)] ${\rm wt}(\tilde{e}_i b) = {\rm wt}(b) + \alpha_i$, $\varepsilon_i (\tilde{e}_i b) = \varepsilon_i (b) - 1$, and $\varphi_i(\tilde{e}_i b) = \varphi_i(b) +1$ if $\tilde{e}_i b \in \mathcal{B}$,
\item[(iii)] ${\rm wt}(\tilde{f}_i b) = {\rm wt}(b) - \alpha_i$, $\varepsilon_i (\tilde{f}_i b) = \varepsilon_i(b) +1$, and $\varphi_i(\tilde{f}_i b) = \varphi_i(b) - 1$ if $\tilde{f}_i b \in \mathcal{B}$,
\item[(iv)] $b^\prime = \tilde{e}_i b$ if and only if $b = \tilde{f}_i b^\prime$,
\item[(v)] $\tilde{e}_i b = \tilde{f}_i b = 0$ if $\varphi_i(b) = -\infty$.
\end{enumerate}
\end{defi}

\begin{defi}[{see \cite[Section 1.2]{Kas4}}]
For crystals $\mathcal{B}_1$ and $\mathcal{B}_2$, a map $\psi \colon \mathcal{B}_1 \cup \{0\} \rightarrow \mathcal{B}_2 \cup \{0\}$ is called a \emph{strict morphism} of crystals from $\mathcal{B}_1$ to $\mathcal{B}_2$ if it satisfies the following conditions:
\begin{enumerate}
\item[(i)] $\psi(0) = 0$,
\item[(ii)] ${\rm wt}(\psi(b)) = {\rm wt}(b)$, $\varepsilon_i(\psi(b)) = \varepsilon_i(b)$, and $\varphi_i(\psi(b)) = \varphi_i(b)$ for $i \in I$ and $b \in \mathcal{B}_1$ such that $\psi(b) \in \mathcal{B}_2$,
\item[(iii)] $\tilde{e}_i \psi(b) = \psi(\tilde{e}_i b)$ and $\tilde{f}_i \psi(b) = \psi(\tilde{f}_i b)$ for $i \in I$ and $b \in \mathcal{B}_1$,
\end{enumerate}
where we set $\tilde{e}_i \psi(b) = \tilde{f}_i \psi(b) = 0$ if $\psi(b) = 0$. 
If $\psi \colon \mathcal{B}_1 \cup \{0\} \rightarrow \mathcal{B}_2 \cup \{0\}$ is injective in addition, then we call it a \emph{strict embedding} of crystals. 
\end{defi}

For a Borel subgroup $B \subseteq G$ containing $H$, the quotient space $G/B$ is called the \emph{full flag variety}, which is a nonsingular projective variety. 
Let $P_+ \subseteq P$ be the set of dominant integral weights, $P_{++} \subseteq P_+$ the set of regular dominant integral weights, and $c_{i, j} \coloneqq \langle \alpha_j, h_i \rangle$ the Cartan integer for $i, j \in I$.
For $\lambda \in P_+$, we denote by $V(\lambda)$ the irreducible highest weight $G$-module over $\c$ with highest weight $\lambda$ and with highest weight vector $v_{\lambda}$. 
For $\lambda \in P_+$, define a line bundle $\mathcal{L}_\lambda$ on $G/B$ by
\[\mathcal{L}_\lambda \coloneqq (G \times \mathbb{C})/B,\] 
where the right $B$-action on $G \times \mathbb{C}$ is given by $(g, c) \cdot b \coloneqq (g b, \lambda(b) c)$ for $g \in G$, $c \in \mathbb{C}$, and $b \in B$. 
For a closed subvariety $Z \subseteq G/B$, the restriction of $\mathcal{L}_\lambda$ to $Z$ is also denoted by the same symbol $\mathcal{L}_\lambda$. 

\begin{prop}[{see, for instance, \cite[Sections I\hspace{-.1em}I.4.4 and I\hspace{-.1em}I.8.5]{Jan}}]
For $\lambda \in P_+$, the line bundle $\mathcal{L}_\lambda$ on $G/B$ is generated by global sections. 
In addition, it is very ample if and only if $\lambda \in P_{++}$.
\end{prop}

For $\lambda \in P_+$, we know by the Borel--Weil theorem that the space $H^0(G/B, \mathcal{L}_\lambda)$ of global sections is a $G$-module which is isomorphic to the dual module $V(\lambda)^\ast \coloneqq {\rm Hom}_\mathbb{C}(V(\lambda), \mathbb{C})$.
Let $B^- \subseteq G$ denote the Borel subgroup opposite to $B$, and $U^-$ the unipotent radical of $B^-$.
Lusztig \cite{Lus_can, Lus_quivers, Lus1} and Kashiwara \cite{Kas1, Kas2, Kas3} constructed specific $\c$-bases of $\c[U^-]$ and $H^0(G/B, \mathcal{L}_\lambda)$ via the quantized enveloping algebra $U_q(\mathfrak{g})$ associated with $\mathfrak{g}$. 
These are called (the specialization at $q = 1$ of) the \emph{dual canonical bases} ($=$ the \emph{upper global bases}), and denoted by $\{G^{\rm up}(b) \mid b \in \mathcal{B}(\infty)\} \subseteq \c[U^-]$ and by $\{G_{\lambda}^{\rm up}(b) \mid b \in \mathcal{B}(\lambda)\} \subseteq H^0(G/B, \mathcal{L}_\lambda)$, respectively. 
The index sets $\mathcal{B}(\infty)$ and $\mathcal{B}(\lambda)$ are typical examples of crystals, called \emph{crystal bases}.
See \cite{Kas5} for a survey on crystal bases. 
An element $b_\infty \in \mathcal{B}(\infty)$ (resp., $b_\lambda \in \mathcal{B}(\lambda)$) is uniquely determined by the condition that $G^{\rm up}(b_\infty) \in \c[U^-]$ is a constant function on $U^-$ (resp., $G_{\lambda}^{\rm up}(b_\lambda)$ is a lowest weight vector in $H^0(G/B, \mathcal{L}_\lambda)$).

\begin{thm}[{see \cite[Theorem 5]{Kas2}}]\label{p:relations_between_crystals}
For $\lambda \in P_+$, there exists a unique surjective map 
\[\pi_\lambda \colon \mathcal{B} (\infty) \twoheadrightarrow \mathcal{B} (\lambda) \cup \{0\}\]
such that $\pi_\lambda (b_\infty) = b_\lambda$ and $\pi_\lambda (\tilde{f}_i b) = \tilde{f}_i \pi_\lambda (b)$ for all $i \in I$ and $b \in \mathcal{B}(\infty)$.
In addition, for
\[\widetilde{\mathcal{B}} (\lambda) \coloneqq \{b \in \mathcal{B}(\infty) \mid \pi_\lambda(b) \neq 0\},\] 
the restriction map $\pi_\lambda \colon \widetilde{\mathcal{B}} (\lambda) \rightarrow \mathcal{B} (\lambda)$ is bijective. 
\end{thm}

The natural projection $G \twoheadrightarrow G/B$ induces an open embedding $U^- \hookrightarrow G/B$ by which $U^-$ is identified with an affine open subvariety of $G/B$. 

\begin{prop}[{see, for instance, \cite[Proposition 3.18]{FO}}]\label{p:relation_between_crystals_G/B}
Let $\lambda \in P_+$, and set $\tau_\lambda \coloneqq G_{\lambda}^{\rm up}(b_\lambda) \in H^0(G/B, \mathcal{L}_\lambda)$.
\begin{enumerate}
\item[{\rm (1)}] The section $\tau_\lambda$ does not vanish on $U^-\ (\hookrightarrow G/B)$;
in particular, the restriction $(\sigma/\tau_\lambda)|_{U^-}$ of a rational function $\sigma/\tau_\lambda \in \c(G/B)$ belongs to $\c[U^-]$ for all $\sigma \in H^0(G/B, \mathcal{L}_\lambda)$.
\item[{\rm (2)}] A map $\iota_\lambda \colon H^0(G/B, \mathcal{L}_\lambda) \rightarrow \c[U^-]$ given by $\iota_\lambda (\sigma) \coloneqq (\sigma/\tau_\lambda)|_{U^-}$ is injective.
\item[{\rm (3)}] The equality $G^{\rm up}(b) = \iota_\lambda (G_{\lambda}^{\rm up}(\pi_\lambda(b)))$ holds for all $b \in \widetilde{\mathcal{B}} (\lambda)$. 
\end{enumerate}
\end{prop}

Let $N_G(H)$ be the normalizer of $H$ in $G$, $W \coloneqq N_G(H)/H$ the Weyl group of $\mathfrak{g}$, and $\widetilde{w} \in N_G(H)$ a lift for $w \in W = N_G(H)/H$.  
The group $W$ is generated by the set $\{s_i \mid i \in I\}$ of simple reflections. 
A sequence ${\bm i} = (i_1, \ldots, i_m) \in I^m$ is called a \emph{reduced word} for $w \in W$ if $w = s_{i_1} \cdots s_{i_m}$ and if $m$ is the minimum among such expressions of $w$. 
In this case, the length $m$ is called the \emph{length} of $w$, which is denoted by $\ell(w)$. 
Let $R(w)$ be the set of reduced words for $w \in W$, and $w_0 \in W$ the longest element.
For $w \in W$, we set 
\[X_w \coloneqq \overline{B \widetilde{w} B/B} \subseteq G/B\quad ({\rm resp.}, X^w \coloneqq \overline{B^- \widetilde{w} B/B} \subseteq G/B),\]
which is called a \emph{Schubert variety} (resp., an \emph{opposite Schubert variety}). 
These varieties $X_w$ and $X^w$ are irreducible normal projective varieties and independent of the choice of a lift $\widetilde{w}$ for $w$. 
In addition, we have 
\[\widetilde{w_0} X_w = \overline{\widetilde{w_0} B \widetilde{w} B/B} = \overline{B^- \widetilde{w_0} \widetilde{w} B/B} = X^{w_0 w}.\]
Let $\leq$ denote the Bruhat order on $W$.

\begin{prop}[{see, for instance, \cite[Section 1.3]{Bri}}]
For $v, w \in W$, the intersection $X_w \cap X^v$ is nonempty if and only if $v \leq w$. 
In this case, the scheme-theoretic intersection 
\[X_w^v \coloneqq X_w \cap X^v \subseteq G/B\]
is reduced and irreducible. 
In addition, it holds that $\dim_\c (X_w^v) = \ell(w) - \ell(v)$. 
\end{prop}

The closed subvariety $X_w^v$ of $G/B$ is called a \emph{Richardson variety}. 
Let $e \in W$ denote the identity element.
Since both $X_{w_0}$ and $X^e$ coincide with $G/B$, it follows that $X_{w_0}^w = X^w$ and $X_w^e = X_w$ for each $w \in W$.
For $w \in W$ and $\lambda \in P_+$, define a $B$-submodule $V_w(\lambda) \subseteq V(\lambda)$ (resp., a $B^-$-submodule $V^w(\lambda) \subseteq V(\lambda)$) by
\begin{align*}
V_w(\lambda) \coloneqq \sum_{b \in B} \c b \widetilde{w} v_{\lambda}\quad({\rm resp.},\ V^w(\lambda) \coloneqq \sum_{b \in B^-} \c b \widetilde{w} v_{\lambda}),
\end{align*}
which is called a \emph{Demazure module} (resp., an \emph{opposite Demazure module}).  
Then we see that $\widetilde{w_0} V_w(\lambda) = V^{w_0 w}(\lambda)$.
By the Borel--Weil type theorem (see, for instance, \cite[Corollary 8.1.26]{Kum}), the space $H^0(X_w, \mathcal{L}_\lambda)$ of global sections is a $B$-module which is isomorphic to the dual module $V_w (\lambda)^\ast \coloneqq {\rm Hom}_\mathbb{C}(V_w(\lambda), \mathbb{C})$.
From this, it follows that
\begin{align*}
H^0(X^w, \mathcal{L}_\lambda) = H^0(\widetilde{w_0} X_{w_0 w}, \mathcal{L}_\lambda) \simeq (\widetilde{w_0} V_{w_0 w}(\lambda))^\ast = V^w (\lambda)^\ast
\end{align*}
as $B^-$-modules. 
Let $\pi_w \colon H^0(G/B, \mathcal{L}_\lambda) \rightarrow H^0(X_w, \mathcal{L}_\lambda)$ (resp., $\pi^w \colon H^0(G/B, \mathcal{L}_\lambda) \rightarrow H^0(X^w, \mathcal{L}_\lambda)$) denote the restriction map.

\begin{prop}[{\cite[Proposition 3.2.3 (i) and equation (4.1)]{Kas4}}]\label{p:characterization_of_Demazure_crystals}
For $w \in W$ and $\lambda \in P_+$, there exists a unique subset $\mathcal{B}_w(\lambda)$ of $\mathcal{B}(\lambda)$ such that $\{\pi_w (G^{\rm up} _\lambda(b)) \mid b \in \mathcal{B}_w(\lambda)\}$ forms a $\mathbb{C}$-basis of $H^0(X_w, \mathcal{L}_\lambda)$ and such that $\pi_w (G^{\rm up} _\lambda(b)) = 0$ for all $b \in \mathcal{B}(\lambda) \setminus \mathcal{B}_w(\lambda)$.
Similarly, there exists a unique subset $\mathcal{B}^w(\lambda)$ of $\mathcal{B}(\lambda)$ such that analogous conditions hold for $\pi^w$.
\end{prop}

These subsets $\mathcal{B}_w(\lambda)$ and $\mathcal{B}^w(\lambda)$ of $\mathcal{B}(\lambda)$ are called a \emph{Demazure crystal} and an \emph{opposite Demazure crystal}, respectively. 

\begin{prop}[{\cite[Propositions 3.2.3 (ii), (iii) and 4.2]{Kas4}}]\label{p:properties of Demazure}
Let $w \in W$, $\lambda \in P_+$, and $i \in I$.
\begin{enumerate}
\item[{\rm (1)}] It holds that $\tilde{e}_i \mathcal{B}_w(\lambda) \subseteq \mathcal{B}_w(\lambda) \cup \{0\}$ and $\tilde{f}_i \mathcal{B}^w(\lambda) \subseteq \mathcal{B}^w(\lambda) \cup \{0\}$.  
\item[{\rm (2)}] If $s_i w < w$, then  
\begin{align*}
\mathcal{B}_w(\lambda) = \bigcup_{k \in \z_{\geq 0}} \tilde{f}_i ^k \mathcal{B}_{s_i w}(\lambda) \setminus \{0\}\quad {\it and}\quad \mathcal{B}^{s_i w}(\lambda) = \bigcup_{k \in \z_{\geq 0}} \tilde{e}_i ^k \mathcal{B}^w(\lambda) \setminus \{0\}.
\end{align*}
\end{enumerate}
\end{prop}

For $v, w \in W$ such that $v \leq w$, let $\pi_w^v \colon H^0(G/B, \mathcal{L}_\lambda) \rightarrow H^0(X_w^v, \mathcal{L}_\lambda)$ denote the restriction map. 
Then the kernel of $\pi_w^v$ coincides with the sum of kernels of $\pi_w$ and $\pi^v$. 
Hence we have
\begin{align*}
H^0(X_w^v, \mathcal{L}_\lambda) \simeq (V_w (\lambda) \cap V^v (\lambda))^\ast
\end{align*}
as $H$-modules.
In addition, the set $\{\pi_w^v (G^{\rm up} _\lambda(b)) \mid b \in \mathcal{B}_w(\lambda) \cap \mathcal{B}^v(\lambda)\}$ forms a $\mathbb{C}$-basis of $H^0(X_w^v, \mathcal{L}_\lambda)$ and the equality $\pi_w^v (G^{\rm up} _\lambda(b)) = 0$ holds for all $b \in \mathcal{B}(\lambda) \setminus (\mathcal{B}_w(\lambda) \cap \mathcal{B}^v(\lambda))$.

\subsection{String polytopes and Nakashima--Zelevinsky polytopes}

In this subsection, we recall some polyhedral parametrizations of crystal bases, following \cite{BZ, Kas4, Lit, Nak, NZ}. 
Set $N \coloneqq \ell(w_0) = \dim_\c (G/B)$.

\begin{defi}[{see \cite[Section 1]{Lit}}]\label{d:string_parametrization}
Let ${\bm i} = (i_1, i_2, \ldots, i_N) \in R(w_0)$, and $\lambda \in P_+$.
For $b \in \mathcal{B}(\lambda)$, define $\Phi_{\bm i} (b) = (a_1, a_2, \ldots, a_N) \in \mathbb{Z}_{\geq 0}^N$ by 
\[a_k \coloneqq \max\{a \in \mathbb{Z}_{\geq 0} \mid \tilde{e}_{i_k} ^a \tilde{e}_{i_{k-1}} ^{a_{k-1}} \cdots \tilde{e}_{i_1} ^{a_1} b \neq 0\}\quad {\rm for}\ 1 \leq k \leq N.\]
It is called \emph{Berenstein--Littelmann--Zelevinsky's string parametrization} associated with ${\bm i}$. 
\end{defi}

The map $\Phi_{\bm i} \colon \mathcal{B}(\lambda) \rightarrow \mathbb{Z}_{\geq 0} ^N$ is indeed injective. 

\begin{defi}[{see \cite[Definition 3.5]{Kav} and \cite[Section 1]{Lit}}]\label{d:string_polytopes}
For ${\bm i} \in R(w_0)$ and $\lambda \in P_+$, define a subset $\mathcal{S}_{\bm i} (\lambda) \subseteq \mathbb{Z}_{>0} \times \mathbb{Z}^N$ by 
\[\mathcal{S}_{\bm i} (\lambda) \coloneqq \bigcup_{k \in \mathbb{Z}_{>0}} \{(k, \Phi_{\bm i}(b)) \mid b \in \mathcal{B}(k\lambda)\}.\] 
We denote by $\mathcal{C}_{\bm i} (\lambda) \subseteq \mathbb{R}_{\geq 0} \times \mathbb{R}^N$ the smallest real closed cone containing $\mathcal{S}_{\bm i} (\lambda)$. 
Then let us define a subset $\Delta_{\bm i} (\lambda) \subseteq \mathbb{R}^N$ by 
\[\Delta_{\bm i} (\lambda) \coloneqq \{{\bm a} \in \mathbb{R}^N \mid (1, {\bm a}) \in \mathcal{C}_{\bm i} (\lambda)\},\] 
which is called \emph{Berenstein--Littelmann--Zelevinsky's string polytope}.
\end{defi}

\begin{prop}[{see \cite[Section 3.2]{BZ} and \cite[Section 1]{Lit}}]\label{string lattice points}
For ${\bm i} \in R(w_0)$ and $\lambda \in P_+$, the string polytope $\Delta_{\bm i} (\lambda)$ is a rational convex polytope, and it holds that $\Delta_{\bm i} (\lambda) \cap \mathbb{Z}^N = \Phi_{\bm i} (\mathcal{B}(\lambda))$.
\end{prop}

Caldero \cite{Cal} constructed a family of toric degenerations of $G/B$ using string polytopes. 
These toric degenerations were geometrically interpreted by Kaveh \cite{Kav} from the theory of Newton--Okounkov bodies.

\begin{thm}[{see \cite[Section 0.2]{Cal} and \cite[Corollary 4.2]{Kav}}]
For ${\bm i} \in R(w_0)$ and $\lambda \in P_+$, the string polytope $\Delta_{\bm i} (\lambda)$ is identical to a Newton--Okounkov body $\Delta(G/B, \mathcal{L}_\lambda, v, \tau)$ of $(G/B, \mathcal{L}_\lambda)$ associated with some $v$ and $\tau$. 
In addition, if $\lambda \in P_{++}$, then there exists a toric degeneration of $G/B$ to the irreducible normal projective toric variety $X(\Delta_{\bm i} (\lambda))$ corresponding to $\Delta_{\bm i} (\lambda)$.
\end{thm}

Fix ${\bm i} = (i_1, i_2, \ldots, i_N) \in R(w_0)$, and consider a sequence ${\bm j} = (\ldots, j_k, \ldots, j_{N+1}, j_N, \ldots, j_1)$ of elements in $I$ such that $j_k = i_{N-k+1}$ for $1 \leq k \leq N$, $j_k \neq j_{k+1}$ for all $k \in \z_{\geq 1}$, and the cardinality of $\{k \in \z_{\geq 1} \mid j_k = i\}$ is $\infty$ for every $i \in I$. 
Following \cite{Kas4, NZ}, we associate to ${\bm j}$ a crystal structure on 
\[\z^{\infty} \coloneqq \{(\ldots, a_k, \ldots, a_2, a_1) \mid a_k \in \mathbb{Z}\ {\rm for}\ k \in \z_{\geq 1}\ {\rm and}\ a_k = 0\ {\rm for}\ k \gg 0\}\] 
as follows. 
For $k \in \z_{\geq 1}$, $i \in I$, and ${\bm a} = (\ldots, a_\ell, \ldots, a_2, a_1) \in \z^\infty$, set 
\begin{align*}
&\sigma_k({\bm a}) \coloneqq a_k + \sum_{\ell > k} c_{j_k, j_\ell} a_\ell \in \mathbb{Z},\\
&\sigma^{(i)}({\bm a}) \coloneqq \max\{\sigma_\ell({\bm a}) \mid \ell \in \z_{\geq 1},\ j_\ell = i\} \in \mathbb{Z},\ {\rm and}\\
&M^{(i)}({\bm a}) \coloneqq \{\ell \in \z_{\geq 1} \mid j_\ell = i,\ \sigma_\ell({\bm a}) = \sigma^{(i)}({\bm a})\}.
\end{align*}
Since $a_\ell = 0$ for $\ell \gg 0$, the integers $\sigma_k({\bm a}), \sigma^{(i)}({\bm a})$ are well-defined, and we have $\sigma^{(i)}({\bm a}) \geq 0$. 
In addition, if $\sigma^{(i)}({\bm a}) > 0$, then $M^{(i)}({\bm a})$ is a finite set. 
Define a crystal structure on $\z^\infty$ by 
\begin{align*}
&{\rm wt}({\bm a}) \coloneqq - \sum_{k = 1} ^\infty a_k \alpha_{j_k},\ \varepsilon_i({\bm a}) \coloneqq \sigma^{(i)}({\bm a}),\ \varphi_i({\bm a}) \coloneqq \varepsilon_i ({\bm a}) + \langle {\rm wt}({\bm a}), h_i \rangle,\ {\rm and}\\
&\tilde{e}_i {\bm a} \coloneqq 
\begin{cases}
(a_k - \delta_{k, \max M^{(i)}({\bm a})})_{k \in \z_{\geq 1}} &{\rm if}\ \sigma^{(i)}({\bm a}) > 0,\\
0 &{\rm otherwise},
\end{cases}\\
&\tilde{f}_i {\bm a} \coloneqq (a_k + \delta_{k, \min M^{(i)}({\bm a})})_{k \in \z_{\geq 1}} 
\end{align*}
for $i \in I$ and ${\bm a} = (\ldots, a_k, \ldots, a_2, a_1) \in \mathbb{Z}^\infty$, where $\delta_{k, \ell}$ is the Kronecker delta; we denote this crystal by $\z^\infty _{\bm j}$. 

\begin{prop}[{see \cite[Section 2.4]{NZ}}]\label{p:polyhedral realizations}
There exists a unique strict embedding of crystals $\widetilde{\Psi}_{\bm j} \colon \mathcal{B} (\infty) \lhook\joinrel\rightarrow \mathbb{Z}^\infty _{\bm j}$ such that $\widetilde{\Psi}_{\bm j} (b_\infty) = (\ldots, 0, \ldots, 0, 0)$.
In addition, if $(\ldots, a_k, \ldots, a_2, a_1) \in \widetilde{\Psi}_{\bm j} (\mathcal{B} (\infty))$, then $a_k = 0$ for all $k > N$.
\end{prop}

The embedding $\widetilde{\Psi}_{\bm j}$ is called the \emph{Kashiwara embedding} of $\mathcal{B}(\infty)$ with respect to ${\bm j}$. 

\begin{defi}[{see \cite[Theorem 3.2]{Nak} and \cite[Definition 3.8]{Fuj}}]
Define $\Psi_{\bm i} \colon \mathcal{B}(\lambda) \hookrightarrow \mathbb{Z}^N$, $b \mapsto (a_1, a_2, \ldots, a_N)$, by 
\[\widetilde{\Psi}_{\bm j} (\tilde{b}) = (\ldots, 0, 0, a_1, a_2, \ldots, a_N),\]
where $\tilde{b}$ is an element of $\widetilde{\mathcal{B}}(\lambda)$ determined by $\pi_\lambda (\tilde{b}) = b$; the map $\Psi_{\bm i}$ is also called a \emph{Kashiwara embedding}.
\end{defi}

The embedding $\Psi_{\bm i}$ is independent of the choice of an extension ${\bm j}$ of ${\bm i}$ by \cite[Section 2.4]{NZ}.

\begin{defi}[{see \cite[Definition 2.15]{FN} and \cite[Definition 3.9]{Fuj}}]
Let ${\bm i} \in R(w_0)$, and $\lambda \in P_+$. 
Replacing $\Phi_{\bm i}$ by $\Psi_{\bm i}$ in the definitions of $\mathcal{S}_{\bm i} (\lambda), \mathcal{C}_{\bm i} (\lambda), \Delta_{\bm i} (\lambda)$ in \cref{d:string_polytopes}, we define $\widetilde{\mathcal{S}}_{\bm i} (\lambda), \widetilde{\mathcal{C}}_{\bm i} (\lambda), \widetilde{\Delta}_{\bm i} (\lambda)$, respectively. 
The set $\widetilde{\Delta}_{\bm i} (\lambda)$ is called a \emph{Nakashima--Zelevinsky polytope}. 
\end{defi}

\begin{prop}[{\cite[Corollaries 2.20 and 4.3]{FN}}]\label{p:convexity}
For ${\bm i} \in R(w_0)$ and $\lambda \in P_+$, the Nakashima--Zelevinsky polytope $\widetilde{\Delta}_{\bm i} (\lambda)$ is a rational convex polytope, and the equality $\widetilde{\Delta}_{\bm i} (\lambda) \cap \z^N = \Psi_{\bm i} (\mathcal{B}(\lambda))$ holds.
\end{prop}

The author and Naito \cite{FN} realized Nakashima--Zelevinsky polytopes as Newton--Okounkov bodies, which induce toric degenerations of $G/B$. 

\begin{thm}[{see \cite[Corollary 4.2]{FN} and \cite[Theorem 6.5]{Fuj}}]
For ${\bm i} \in R(w_0)$ and $\lambda \in P_+$, the Nakashima--Zelevinsky polytope $\widetilde{\Delta}_{\bm i} (\lambda)$ is identical to a Newton--Okounkov body $\Delta(G/B, \mathcal{L}_\lambda, v, \tau)$ of $(G/B, \mathcal{L}_\lambda)$ associated with some $v$ and $\tau$. 
In addition, if $\lambda \in P_{++}$, then there exists a toric degeneration of $G/B$ to the irreducible normal projective toric variety $X(\widetilde{\Delta}_{\bm i} (\lambda))$ corresponding to $\widetilde{\Delta}_{\bm i} (\lambda)$.
\end{thm}

\begin{ex}
Let $G = SL_3(\c)$, ${\bm i} = (1, 2, 1)$, and $\rho \coloneqq \alpha_1 + \alpha_2 \in P_{++}$. 
Then the string parametrizations $\Phi_{\bm i} (b)$, $b \in \mathcal{B}(\rho)$, are given as follows: 
\[\xymatrix{
 & {(1, 0, 0)} \ar[r]^-{2} & {(0, 1, 1)} \ar[r]^-{2} & {(0, 2, 1)} \ar[dr]^-{1}\\
{(0, 0, 0)} \ar[ur]^-{1} \ar[dr]^-{2} & &  & & {(1, 2, 1)},\\
 & {(0, 1, 0)} \ar[r]^-{1} & {(1, 1, 0)} \ar[r]^-{1} & {(2, 1, 0)} \ar[ur]^-{2} & 
}\]
where $\Phi_{\bm i} (b) \xrightarrow{i} \Phi_{\bm i} (b^\prime)$ if and only if $b^\prime = \tilde{f}_i b$.
These eight integer vectors form the set of lattice points in the string polytope $\Delta_{\bm i} (\rho)$. 
Indeed, $\Delta_{\bm i} (\rho)$ is given by 
\[\Delta_{\bm i} (\rho) = \{(a_1, a_2, a_3) \in \r^3 \mid 0 \leq a_3 \leq 1,\ a_3 \leq a_2 \leq a_3 + 1,\ 0 \leq a_1 \leq a_2 - 2a_3 + 1\};\] 
see Figure \ref{figure_ex_string_polytope}.

\begin{figure}[!ht]
\begin{center}
   \includegraphics[width=8.0cm,bb=60mm 180mm 150mm 230mm,clip]{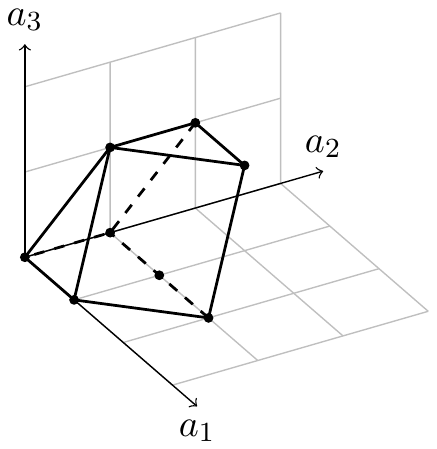}
	\caption{The string polytope $\Delta_{\bm i} (\rho)$.}
	\label{figure_ex_string_polytope}
\end{center}
\end{figure}

\noindent In addition, the values $\Psi_{\bm i} (b)$, $b \in \mathcal{B}(\rho)$, of the Kashiwara embedding $\Psi$ are given as follows: 
\[\xymatrix{
 & {(0, 0, 1)} \ar[r]^-{2} & {(0, 1, 1)} \ar[r]^-{2} & {(0, 2, 1)} \ar[dr]^-{1}\\
{(0, 0, 0)} \ar[ur]^-{1} \ar[dr]^-{2} & &  & & {(1, 2, 1)},\\
 & {(0, 1, 0)} \ar[r]^-{1} & {(1, 1, 0)} \ar[r]^-{1} & {(1, 1, 1)} \ar[ur]^-{2} & 
}\]
where $\Psi_{\bm i} (b) \xrightarrow{i} \Psi_{\bm i} (b^\prime)$ if and only if $b^\prime = \tilde{f}_i b$.
These eight integer vectors form the set of lattice points in the Nakashima--Zelevinsky polytope $\widetilde{\Delta}_{\bm i} (\rho)$. 
Indeed, $\widetilde{\Delta}_{\bm i} (\rho)$ is given by 
\[\widetilde{\Delta}_{\bm i} (\rho) = \{(a_1, a_2, a_3) \in \r^3 \mid 0 \leq a_1 \leq 1,\ 0 \leq a_3 \leq 1,\ a_1 \leq a_2 \leq a_3 + 1\};\] 
see Figure \ref{figure_ex_NZ_polytope}.

\begin{figure}[!ht]
\begin{center}
   \includegraphics[width=8.0cm,bb=60mm 180mm 150mm 230mm,clip]{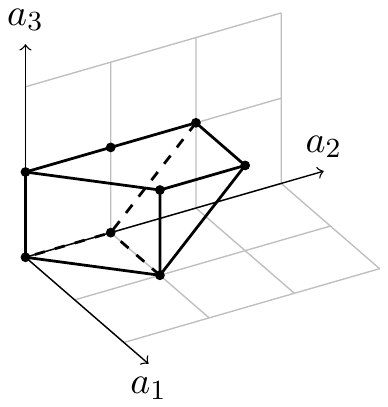}
	\caption{The Nakashima--Zelevinsky polytope $\widetilde{\Delta}_{\bm i} (\rho)$.}
	\label{figure_ex_NZ_polytope}
\end{center}
\end{figure}

\end{ex}

\subsection{Morier-Genoud's semi-toric degenerations of Richardson varieties}\label{ss:semi-toric}

In this subsection, we review some results in \cite{Mor}.
For ${\bm i} \in R(w_0)$, define a subset $\mathscr{S}_{\bm i} \subseteq P_+ \times \z^N$ by 
\[\mathscr{S}_{\bm i} \coloneqq \bigcup_{\lambda \in P_+} \{(\lambda, \Phi_{\bm i} (b)) \mid b \in \mathcal{B}(\lambda)\},\] 
and let $\mathscr{C}_{\bm i} \subseteq P_\r \times \r^N$ denote the smallest real closed cone containing $\mathscr{S}_{\bm i}$, where $P_\r \coloneqq P \otimes_\mathbb{Z} \mathbb{R}$. 

\begin{prop}[{see \cite[Section 3.2]{BZ} and \cite[Section 1]{Lit}}] 
The real closed cone $\mathscr{C}_{\bm i}$ is a rational convex polyhedral cone, and it holds that $\mathscr{C}_{\bm i} \cap (P_+ \times \z^N) = \mathscr{S}_{\bm i}$.
\end{prop}

For $\lambda \in P_+$, it holds that
\begin{align*}
\mathscr{S}_{\bm i} \cap \pi_1 ^{-1} (\z_{>0} \lambda) = \omega_\lambda(\mathcal{S}_{\bm i} (\lambda))\quad {\rm and}\quad \mathscr{C}_{\bm i} \cap \pi_1 ^{-1} (\r_{\geq 0} \lambda) = \omega_\lambda(\mathcal{C}_{\bm i} (\lambda)),
\end{align*}
where let $\pi_1 \colon P_\r \times \r^N \rightarrow P_\r$ denote the first projection, and $\omega_\lambda \colon \r_{\geq 0} \times \r^N \rightarrow P_\r \times \r^N$ is defined by $\omega_\lambda (k, {\bm a}) \coloneqq (k\lambda, {\bm a})$.
In particular, we have $\pi_2(\mathscr{C}_{\bm i} \cap \pi_1 ^{-1} (\lambda)) = \Delta_{\bm i} (\lambda)$ for the second projection $\pi_2 \colon P_\r \times \r^N \rightarrow \r^N$.
For $v, w \in W$ such that $v \leq w$, define a subset $\mathscr{S}_{\bm i}(X_w^v) \subseteq P_+ \times \mathbb{Z}^N$ by 
\[\mathscr{S}_{\bm i}(X_w^v) \coloneqq \bigcup_{\lambda \in P_+} \{(\lambda, \Phi_{\bm i} (b)) \mid b \in \mathcal{B}_w(\lambda) \cap \mathcal{B}^v(\lambda)\},\] 
and denote by $\mathscr{C}_{\bm i}(X_w^v) \subseteq P_\mathbb{R} \times \mathbb{R}^N$ the smallest real closed cone containing $\mathscr{S}_{\bm i}(X_w^v)$. 
For $\lambda \in P_+$, we write
\begin{align*}
&\Delta_{\bm i}(\lambda, X_w^v) \coloneqq \pi_2(\mathscr{C}_{\bm i}(X_w^v) \cap \pi_1 ^{-1} (\lambda)).
\end{align*}

\begin{thm}[{see the proof of \cite[Proposition 3.5 and Theorem 3.7]{Mor}}]\label{t:semi-toric_degenerations}
Let ${\bm i} \in R(w_0)$, and $v, w \in W$ such that $v \leq w$.
\begin{enumerate}
\item[{\rm (1)}] The set $\mathscr{C}_{\bm i}(X_w^v)$ is a union of faces of $\mathscr{C}_{\bm i}$, and the following equality holds:
\begin{align*}
\mathscr{C}_{\bm i}(X_w^v) \cap (P_+ \times \z^N) &= \mathscr{S}_{\bm i}(X_w^v).
\end{align*}
\item[{\rm (2)}] For $\lambda \in P_+$, the set $\Delta_{\bm i}(\lambda, X_w^v)$ is a union of faces of $\Delta_{\bm i}(\lambda)$.
\item[{\rm (3)}] For $\lambda \in P_{++}$, the Richardson variety $X_w^v$ degenerates to the union of irreducible closed toric subvarieties of $X(\Delta_{\bm i}(\lambda))$ corresponding to the faces of $\Delta_{\bm i}(\lambda, X_w^v)$.
\end{enumerate}
\end{thm}

\section{Cluster algebras and higher rank valuations}\label{s:Cluster_valuation}

In this section, we review a construction of higher rank valuations using cluster structures, following \cite{FO2}. 
Let us start with recalling the definition of upper cluster algebras of geometric type, following \cite{BFZ, FZ:ClusterIV} but using the notation in \cite{FG, GHKK}. 
Fix a finite set $J$ and a subset $J_{\rm uf} \subseteq J$. 
Let $\mathcal{F} \coloneqq \mathbb{C}(z_j \mid j \in J)$ be the field of rational functions in $|J|$ variables. 
For a $J$-tuple $\mathbf{A} = (A_j)_{j \in J}$ of elements of $\mathcal{F}$ and an integer $J_{\rm uf} \times J$ matrix $\varepsilon = (\varepsilon_{i, j})_{i \in J_{\rm uf}, j \in J}$, the pair ${\mathbf s} = (\mathbf{A}, \varepsilon)$ is called a \emph{Fomin--Zelevinsky seed} (an \emph{FZ-seed} for short) of $\mathcal{F}$ if 
\begin{itemize}
	\item[(i)] $\mathbf{A}$ forms a free generating set of the field $\mathcal{F}$, and
   \item[(ii)] the $J_{\rm uf} \times J_{\rm uf}$ submatrix of $\varepsilon$ is skew-symmetrizable, that is, there exists $(d_i)_{i \in J_{\rm uf}} \in \mathbb{Z}^{J_{\rm uf}}_{>0}$ such that $d_i \varepsilon_{i, j} = -d_j \varepsilon_{j, i}$ for all $i, j \in J_{\rm uf}$.  
\end{itemize}
In this case, $\varepsilon$ is called the \emph{exchange matrix} of ${\mathbf s}$. 
Note that our exchange matrix $\varepsilon$ is transposed to the one in \cite[Section 2]{FZ:ClusterIV}.
Let ${\mathbf s} = (\mathbf{A}, \varepsilon) = ((A_j)_{j \in J}, (\varepsilon_{i, j})_{i \in J_{\rm uf}, j \in J})$ be an FZ-seed of $\mathcal{F}$. 
For $k \in J_{\rm uf}$, we define the \emph{mutation} $\mu_k ({\mathbf s}) = (\mu_k (\mathbf{A}), \mu_k (\varepsilon)) = ((A^\prime _j)_{j \in J}, (\varepsilon^\prime _{i, j})_{i \in J_{\rm uf}, j \in J})$ of ${\mathbf s}$ in direction $k$ as follows:
\begin{align*}
\varepsilon^\prime _{i, j} \coloneqq \begin{cases}
-\varepsilon_{i, j}&\ \text{if}\ i=k\ {\rm or}\ j=k,\\
\varepsilon_{i, j} + {\rm sgn}(\varepsilon_{i, k})[\varepsilon_{i, k} \varepsilon_{k, j}]_+&\ {\rm otherwise},
\end{cases}
\end{align*}
\begin{align*}
A_j ^\prime \coloneqq \begin{cases}
\displaystyle \frac{\prod_{\ell \in J} A_\ell ^{[\varepsilon_{k, \ell}]_+} + \prod_{\ell \in J} A_\ell ^{[-\varepsilon_{k, \ell}]_+}}{A_k} &\ {\rm if}\ j = k,\\
A_j &\ {\rm otherwise}
\end{cases}
\end{align*}
for $i \in J_{\rm uf}$ and $j \in J$, where $[c]_+ \coloneqq \max\{c, 0\}$ for $c \in \r$. 
Then $\mu_k ({\mathbf s})$ is also an FZ-seed of $\mathcal{F}$, and it holds that $\mu_k \circ \mu_k ({\mathbf s}) = {\mathbf s}$. 
FZ-seeds ${\mathbf s}$ and ${\mathbf s}^\prime$ are said to be \emph{mutation equivalent} if there exists a sequence $(k_1, k_2, \ldots, k_j)$ of entries in $J_{\rm uf}$ such that ${\mathbf s}^\prime = \mu_{k_j} \circ \cdots \circ \mu_{k_2} \circ \mu_{k_1} ({\mathbf s})$.

\begin{defi}[{\cite[Definition 3.1.1]{Qin}}]\label{d:dominance_order}
Let ${\mathbf s} = (\mathbf{A}, \varepsilon)$ be an FZ-seed of $\mathcal{F}$, and assume that $\varepsilon$ is of full rank. 
Then we define a partial order $\preceq_\varepsilon$ on $\mathbb{Z}^J$ as follows: for ${\bm a}, {\bm a}^\prime \in \mathbb{Z}^J$,
\[{\bm a} \preceq_{\varepsilon} {\bm a}^\prime\ \text{if and only if}\ {\bm a} = {\bm a}^\prime + {\bm v} \varepsilon\ \text{for some}\ {\bm v}\in \mathbb{Z}_{\geq 0}^{J_{\rm uf}},\]
where elements of $\mathbb{Z}^J$ (resp., $\mathbb{Z}_{\geq 0}^{J_{\rm uf}}$) are regarded as $1 \times J$ (resp., $1 \times J_{\rm uf}$) matrices. 
This partial order $\preceq_{\varepsilon}$ on $\mathbb{Z}^{J}$ is called the \emph{dominance order} associated with $\varepsilon$.
\end{defi}

Let $\mathbb{T}$ be the $|J_{\rm uf}|$-regular tree whose edges are labeled by $J_{\rm uf}$ such that the $|J_{\rm uf}|$-edges emanating from each vertex receive different labels. 
We write $t \overset{k}{\text{---}} t^\prime$ if $t, t^\prime \in \mathbb{T}$ are joined by an edge labeled by $k \in J_{\rm uf}$. 

\begin{ex}
If $J_{\rm uf} = \{1, 2\}$, then the $2$-regular tree $\mathbb{T}$ is given as follows:
\begin{align*}
\begin{xy}
\ar@{-}^-{1} (50,0) *\cir<3pt>{}="B";
(60,0) *\cir<3pt>{}="C"
\ar@{-}^-{2} "C";(70,0) *\cir<3pt>{}="D"
\ar@{-}^-{1} "D";(80,0) *\cir<3pt>{}="E"
\ar@{-}^-{2} "E";(90,0) *\cir<3pt>{}="F"
\ar@{-} "F";(95,0)
\ar@{-} "B";(45,0)
\ar@{.} (95,0);(100,0)^*!U{}
\ar@{.} (45,0);(40,0)^*!U{}
\end{xy}
\end{align*}
\end{ex}

An assignment $\mathcal{S} = \{{\mathbf s}_t\}_{t \in \mathbb{T}} = \{(\mathbf{A}_t, \varepsilon_t)\}_{t \in \mathbb{T}}$ of an FZ-seed ${\mathbf s}_t = (\mathbf{A}_t, \varepsilon_t)$ of $\mathcal{F}$ to each vertex $t \in \mathbb{T}$ is called a \emph{cluster pattern} if $\mu_k ({\mathbf s}_t) = {\mathbf s}_{t^\prime}$ whenever $t \overset{k}{\text{---}} t^\prime$ in $\mathbb{T}$. 
For a cluster pattern $\mathcal{S} = \{{\mathbf s}_t = (\mathbf{A}_t, \varepsilon_t)\}_{t \in \mathbb{T}}$, we write 
\[\mathbf{A}_t = (A_{j; t})_{j \in J}, \quad \varepsilon_t = (\varepsilon_{i, j} ^{(t)})_{i \in J_{\rm uf}, j \in J}.\]

\begin{defi}[{see \cite[Definition 1.6]{BFZ}}]\normalfont
For a cluster pattern $\mathcal{S}$, we set
\[\mathscr{U}(\mathcal{S}) \coloneqq \bigcap_{t \in \mathbb{T}} \c[A_{j; t} ^{\pm 1} \mid j \in J] \subseteq \mathcal{F},\]
which is called the \emph{upper cluster algebra of geometric type}.
\end{defi}

We usually fix a vertex $t_0 \in \mathbb{T}$ and construct a cluster pattern $\mathcal{S} = \{{\mathbf s}_t\}_{t \in \mathbb{T}}$ from one FZ-seed ${\mathbf s}_{t_0}$ of $\mathcal{F}$. 
In this case, we call ${\mathbf s}_{t_0}$ the \emph{initial FZ-seed}. 

\begin{thm}[{\cite[Theorem 3.1]{FZ:ClusterI}}]\label{t:laurentpheno}
Let $\mathcal{S}$ be a cluster pattern. 
Then it holds that 
\[\{A_{j; t} \mid t \in \mathbb{T},\ j \in J\} \subseteq \mathscr{U}(\mathcal{S});\] 
this property is called the \emph{Laurent phenomenon}.
\end{thm}

The elements of $\{A_{j; t} \mid t \in \mathbb{T},\ j \in J\}$ are called \emph{cluster variables} of $\mathscr{U}(\mathcal{S})$.
For $i \in J_{\rm uf}$, set 
\[\widehat{X}_{i; t} \coloneqq \prod_{j \in J} A_{j; t} ^{\varepsilon_{i, j} ^{(t)}};\]
see \cite[Section 3]{FZ:ClusterIV}. 
In the rest of this section, we assume that 
\begin{enumerate}
\item[($\dagger$)] the exchange matrix $\varepsilon_{t_0}$ of ${\mathbf s}_{t_0}$ is of full rank for some $t_0 \in \mathbb{T}$.
\end{enumerate}
This implies that the exchange matrix $\varepsilon_t$ of ${\mathbf s}_t$ is of full rank for all $t \in \mathbb{T}$ (see \cite[Lemma 3.2]{BFZ}). 

\begin{defi}[{see \cite{FZ:ClusterIV, FO2, Qin, Tra}}]\label{d:pointed_elements}
Let $\mathcal{S} = \{{\mathbf s}_t = (\mathbf{A}_t, \varepsilon_t)\}_{t \in \mathbb{T}}$ be a cluster pattern, and fix $t \in \mathbb{T}$. 
A Laurent polynomial $f \in \c[A_{j; t} ^{\pm 1} \mid j \in J]$ is said to be \emph{weakly pointed} at $(g_j)_{j \in J} \in \z^J$ for $t$ if we have
\[f = \left(\prod_{j \in J} A_{j; t} ^{g_j}\right) \left(\sum_{{\bm a}=(a_j)_{j} \in \mathbb{Z}_{\geq 0}^{J_{\rm uf}}} c_{\bm a} \prod_{j \in J_{\rm uf}} \widehat{X}_{j; t} ^{a_j}\right)\]
for some $\{c_{\bm a} \in \c \mid {\bm a} \in \mathbb{Z}_{\geq 0}^{J_{\rm uf}}\}$ such that $c_0 \neq 0$ (see \cite[Definition 3.6]{FO2}). 
In this case, we write $g_t (f) = g_{{\mathbf s}_t} (f) \coloneqq (g_j)_{j \in J} \in \z^J$, which is called the \emph{extended $g$-vector} of $f$ (see \cite[Section 6]{FZ:ClusterIV} and \cite[Definition 3.7]{Tra}). 
If $c_0 = 1$ in addition, then $f$ is said to be \emph{pointed} at $g_t (f)$ (see \cite[Definition 3.1.4]{Qin}).
\end{defi} 

\begin{defi}[{\cite[Definition 3.8]{FO2}}]\normalfont\label{d:main_valuation}
Let $\mathcal{S} = \{{\mathbf s}_t = (\mathbf{A}_t, \varepsilon_t)\}_{t \in \mathbb{T}}$ be a cluster pattern, and fix $t \in \mathbb{T}$. 
We denote by $\preceq_{\varepsilon_t} ^{\rm op}$ the opposite order of the dominance order $\preceq_{\varepsilon_t}$, and take a total order $\leq_t$ on $\z^J$ which refines $\preceq_{\varepsilon_t} ^{\rm op}$.
The total order $\leq_t$ on $\z^J$ induces a total order (denoted by the same symbol $\leq_t$) on the set of Laurent monomials in $A_{j; t}$, $j \in J$, as follows: 
\begin{align*}
\prod_{j \in J} A_{j;t} ^{a_j} \leq_t \prod_{j \in J} A_{j;t} ^{a_j ^\prime}\quad \text{if and only if}\quad (a_j)_{j \in J} \leq_t (a_j ^\prime)_{j \in J}. 
\end{align*}
Then denote by $v_{{\mathbf s}_t}$ (or simply by $v_t$) the lowest term valuation $v^{\rm low} _{\leq_t}$ on $\mathcal{F} = \c(A_{j; t} \mid j \in J)$ with respect to $\leq_t$ (see \cref{ex:lowest_term_valuation}). 
\end{defi} 

If $f \in \c[A_{j; t} ^{\pm 1} \mid j \in J]$ is weakly pointed for $t \in \mathbb{T}$, then we see by the definition of $v_{{\mathbf s}_t}$ that  
\begin{equation}\label{eq:g-vector_valuation}
\begin{aligned}
v_{{\mathbf s}_t} (f) = g_{{\mathbf s}_t} (f)
\end{aligned}
\end{equation}
for every refinement $\leq_t$ of $\preceq_{\varepsilon_t} ^{\rm op}$.
Following \cite[Conjecture 7.12]{FZ:ClusterIV} and \cite[Section 4]{FG}, let us define a \emph{tropicalized cluster mutation} $\mu_k ^T$ at $t \in \mathbb{T}$ for $k \in J_{\rm uf}$ as follows:
\[\mu_k ^T \colon \r^J \rightarrow \r^J,\ (g_j)_{j \in J} \mapsto (g' _j)_{j \in J},\] 
is given by
\begin{equation}\label{eq:tropicalized_cluster}
\begin{aligned}
g' _j \coloneqq
\begin{cases}
g_j + [-\varepsilon_{k, j} ^{(t)}]_+ g_k + \varepsilon_{k, j} ^{(t)} [g_k]_+ &(j \neq k),\\
-g_j &(j = k)
\end{cases}
\end{aligned}
\end{equation}
for $j \in J$.
This is precisely the tropicalization of the mutation in direction $k$ at $t$ for the Fock--Goncharov dual $\mathcal{A}^\vee$ of the $\mathcal{A}$-cluster variety (see \cite[Section 2 and Definition A.4]{GHKK}).

\section{Cluster structures on unipotent cells and toric degenerations}\label{s:unipotent_cell}

\subsection{Cluster structures on unipotent cells}\label{ss:cluster_structure_unipotent_cell}

In this subsection, we review a cluster structure on the coordinate ring of a unipotent cell, which was verified in \cite{BFZ, GLS:Kac-Moody, Dem}.
For $w \in W$, define the \emph{unipotent cell} $U^-_w$ by
\begin{align*}
U^-_w \coloneqq U^- \cap B \widetilde{w} B \subseteq G,
\end{align*}
where $\widetilde{w} \in N_G(H)$ is a lift for $w \in W = N_G(H)/H$. 
The unipotent cell $U^-_w$ is independent of the choice of a lift $\widetilde{w}$ for $w$, and the open embedding $U^- \hookrightarrow G/B$ induces an open embedding $U_w^- \hookrightarrow X_w$.
Let $e_i, f_i, h_i \in \mathfrak{g}$, $i \in I$, be the Chevalley generators, and $\{\varpi_i \mid i \in I\} \subseteq P_+$ the set of fundamental weights. 
For each $i \in I$, denote by $\mathfrak{g}_i$ the Lie subalgebra of $\mathfrak{g}$ generated by $e_i, f_i, h_i$, which is isomorphic to $\mathfrak{sl}_2 (\c)$.
Then the embedding $\mathfrak{g}_i \hookrightarrow \mathfrak{g}$ of Lie algebras is naturally lifted to an algebraic group homomorphism $\psi_i \colon SL_2(\c) \rightarrow G$. 
For $i \in I$, we set
\[\overline{s}_i \coloneqq \psi_i \left(\begin{pmatrix}
0 & -1\\
1 & 0
\end{pmatrix}\right) \in N_G(H),\]
which is a lift for $s_i \in W = N_G(H)/H$.
In addition, for $w \in W$, define its lift $\overline{w} \in N_G(H)$ by $\overline{w} \coloneqq \overline{s}_{i_1} \cdots \overline{s}_{i_m}$, where $(i_1, \ldots, i_m) \in R(w)$.
This is independent of the choice of a reduced word $(i_1, \ldots, i_m) \in R(w)$. 
For $w \in W$ and $\lambda \in P_+$, write $v_{w \lambda} \coloneqq \overline{w} v_\lambda \in V(\lambda)$, and define $f_{w \lambda} \in V(\lambda)^\ast$ by $f_{w \lambda} (v_{w \lambda}) = 1$ and by $f(v) = 0$ for every weight vector $v \in V(\lambda)$ whose weight is different from $w \lambda$. 
For $u, u^\prime \in W$ and $i \in I$, we define a function $\Delta_{u \varpi_i, u^\prime \varpi_i} \in \c[G]$ by 
\[\Delta_{u \varpi_i, u^\prime \varpi_i} (g) \coloneqq f_{u \varpi_i} (g v_{u^\prime \varpi_i}).\]
This function $\Delta_{u \varpi_i, u^\prime \varpi_i}$ is called a \emph{generalized minor} (see also \cite[Section 2.3]{BFZ}). 
For $w \in W$, we denote by $D_{u \varpi_i, u^\prime \varpi_i} \in \c[U_w^-]$ the restriction of $\Delta_{u \varpi_i, u^\prime \varpi_i}$ to $U_w^-$, which is called a \emph{unipotent minor}.
Berenstein--Fomin--Zelevinsky \cite{BFZ} and Williams \cite{Wil} proved that the coordinate ring $\mathbb{C}[G^{w, w'}]$ of a double Bruhat cell $G^{w, w^\prime} \coloneqq B^- \overline{w^\prime} B^- \cap B \overline{w} B$ for $w, w^\prime \in W$ has an upper cluster algebra structure. 
If $w^\prime = e$, then the double Bruhat cell $G^{w, e}$ is closely related to the unipotent cell $U^-_w$, and their result induces an upper cluster algebra structure on $\mathbb{C}[U^-_w]$ (see, for instance, \cite[Appendix B]{FO2}).
Let us see a specific FZ-seed of $\mathbb{C}[U^-_w]$. 
Fix ${\bm i} = (i_1, \ldots, i_m) \in R(w)$. 
For $1 \leq k \leq m$, write $w_{\leq k} \coloneqq s_{i_1} \cdots s_{i_k}$ and
\begin{align*}
k^{+} \coloneqq \min(\{k+1 \leq j \leq m \mid i_j = i_k\} \cup \{m+1\}).
\end{align*} 
We set 
\[J \coloneqq \{1, \ldots, m\},\quad J_{\rm uf} \coloneqq \{j \in J \mid j^+ \neq m+1\},\]
and define an integer $J_{\rm uf} \times J$ matrix $\varepsilon^{\bm i} = (\varepsilon_{s, t})_{s \in J_{\rm uf}, t \in J}$ by
\[\varepsilon_{s, t} \coloneqq
\begin{cases}
1&\text{if}\ s^+ = t, \\
-1&\text{if}\ s = t^+, \\
c_{i_t, i_s}&\text{if}\ s < t < s^+ < t^+, \\
-c_{i_t, i_s}&\text{if}\ t < s < t^+ < s^+,\\
0&\text{otherwise},
\end{cases}\]
where we recall that $(c_{i, j})_{i, j \in I}$ denotes the Cartan matrix.

\begin{ex}
Let $G = SL_4 (\c)$, and ${\bm i} = (1, 2, 1, 3, 2, 1) \in R(w_0)$.
Then the integer matrix $\varepsilon^{\bm i}$ is given by
\[\varepsilon^{\bm i} = \begin{pmatrix}
0 & -1 & 1 & 0 & 0 & 0 \\
1 & 0 & -1 & -1 & 1 & 0 \\
-1 & 1 & 0 & 0 & -1 & 1 
\end{pmatrix}.\]
\end{ex}

For $k \in J$, we set  
\[D(k, {\bm i}) \coloneqq D_{w_{\leq k} \varpi_{i_k}, \varpi_{i_k}},\]
and consider a cluster pattern $\mathcal{S} = \{{\mathbf s}_t = (\mathbf{A}_t, \varepsilon_t)\}_{t \in \mathbb{T}}$ whose initial FZ-seed is given as $\mathbf{s}_{t_0} = ((A_{k; t_0})_{k \in J}, \varepsilon^{\bm i})$. 

\begin{thm}[{\cite[Theorem 2.10]{BFZ} and \cite[Theorem 4.16]{Wil} (see also \cite[Theorem B.4]{FO2})}]\label{t:upper_Bruhat}
For $w \in W$ and ${\bm i} = (i_1, \ldots, i_m) \in R(w)$, there exists a $\mathbb{C}$-algebra isomorphism 
\[\mathscr{U}(\mathcal{S}) \xrightarrow{\sim} \mathbb{C}[U^-_w]\] 
given by $A_{k; t_0} \mapsto D(k, {\bm i})$ for each $k \in J$.
\end{thm}

\begin{rem}
Geiss--Leclerc--Schr\"{o}er \cite{GLS:Kac-Moody} verified the cluster algebra structure on  $\mathbb{C}[U^-_w]$ without going through $G^{w, e}$ under the assumption that the Cartan matrix of $G$ is symmetric. 
Demonet \cite{Dem} extended their method to the case of symmetrizable Kac--Moody groups for specific $w$. 
Goodearl--Yakimov \cite{GY:Mem} showed a quantum cluster algebra structure on a quantum analog of $\mathbb{C}[U^-_w]$ for an arbitrary symmetrizable Kac--Moody group $G$ and $w \in W$. 
They also announced the classical version of their results in  \cite[Section 1.2]{GY:Poisson}. 
\end{rem}

We define an FZ-seed $\mathbf{s}_{\bm i}$ of $\c (U^-_w)$ by 
\[\mathbf{s}_{\bm i} \coloneqq (\mathbf{D}_{\bm i} \coloneqq (D(k, \bm{i}))_{k \in J}, \varepsilon^{\bm i}),\]
which corresponds to the FZ-seed $\mathbf{s}_{t_0} = ((A_{k; t_0})_{k \in J}, \varepsilon^{\bm i})$ under the isomorphism in \cref{t:upper_Bruhat}.

\begin{thm}[{\cite[Theorem 3.5]{FG:amal}}]
For $w \in W$, the cluster pattern associated with $\mathbf{s}_{\bm i}$ is independent of the choice of a reduced word ${\bm i}\in R(w)$, that is, all $\mathbf{s}_{\bm i}$, ${\bm i}\in R(w)$, are mutually mutation equivalent. 
\end{thm}

If $|c_{i, j}| \leq 1$ for all $i, j \in I$, then the entries of $\varepsilon^{\bm i}$ are $0$ or $\pm 1$. Hence $\varepsilon^{\bm i}$ is described as a quiver whose vertex set is $J$ and whose arrow set is given by 
\[\{s \to t \mid \varepsilon_{s, t} = -1\ \text{or}\ \varepsilon_{t, s} = 1\}. \]
Note that there are no arrows between vertices in $J \setminus J_{\rm uf}$.

\begin{ex}
Let $G = SL_4 (\c)$, and ${\bm i} = (1, 2, 1, 3, 2, 1) \in R(w_0)$.
Then the FZ-seed ${\mathbf s}_{\bm i} = ({\mathbf D}_{\bm i}, \varepsilon^{\bm i})$ is described as follows: 

	\hfill
	\scalebox{0.7}[0.7]{
		\begin{xy} 0;<1pt,0pt>:<0pt,-1pt>::
			(180,00) *+{D_{w_0 \varpi_3, \varpi_3}} ="1",
			(120,30) *+{D_{s_1 s_2 \varpi_2, \varpi_2}} ="2",
			(240,30) *+{D_{w_0 \varpi_2, \varpi_2}} ="3",
			(60,60) *+{D_{s_1 \varpi_1, \varpi_1}} ="4",
			(180,60) *+{D_{s_1 s_2 s_1 \varpi_1, \varpi_1}} ="5",
			(300,60) *+{D_{w_0 \varpi_1, \varpi_1}} ="6",
			"2", {\ar"1"},
			"4", {\ar"2"},
			"2", {\ar"5"},
			"5", {\ar"3"},
			"3", {\ar"2"},
			"5", {\ar"4"},
			"6", {\ar"5"},
		\end{xy}
	}
	\hfill
	\hfill

\end{ex}

\subsection{Bases parametrized by tropical points}\label{ss:bases_tropical}

In this subsection, we discuss a $\mathbb{C}$-basis of $\c[U_{w_0} ^-]$ which is parametrized by tropical points. 
Let $\mathcal{S} = \{{\mathbf s}_t = (\mathbf{A}_t, \varepsilon_t)\}_{t \in \mathbb{T}}$ be a cluster pattern for the upper cluster algebra $\c[U_{w_0} ^-]$.
Then we obtain a valuation $v_{{\mathbf s}_t}$ on $\c(U_{w_0} ^-) = \c(G/B)$ for each $t \in \mathbb{T}$ by \cref{d:main_valuation}. 
Recall the injective $\mathbb{C}$-linear map $\iota_\lambda \colon H^0(G/B, \mathcal{L}_{\lambda}) \hookrightarrow \c[U^-]$ given in \cref{p:relation_between_crystals_G/B} for $\lambda \in P_+$.
Since $U^- _{w_0}$ is an open subvariety of $U^-$, the coordinate ring $\c[U^-]$ is regarded as a $\c$-subalgebra of $\c[U^- _{w_0}]$. 
Hence we can think of $\iota_\lambda$ as an injective $\mathbb{C}$-linear map $H^0(G/B, \mathcal{L}_{\lambda}) \hookrightarrow \c[U^- _{w_0}]$.
For $v, w \in W$ such that $v \leq w$, let $\Lambda_{X_w^v}^{(\lambda)} \colon \iota_\lambda(H^0(G/B, \mathcal{L}_{\lambda})) \rightarrow H^0(X_w^v, \mathcal{L}_\lambda)$ denote the composition of the inverse map $\iota_\lambda^{-1} \colon \iota_\lambda(H^0(G/B, \mathcal{L}_{\lambda})) \rightarrow H^0(G/B, \mathcal{L}_{\lambda})$ and the restriction map $H^0(G/B, \mathcal{L}_{\lambda}) \rightarrow H^0(X_w^v, \mathcal{L}_{\lambda})$.
In the present paper, we use a $\mathbb{C}$-basis ${\bf B}_{w_0}$ of $\c[U_{w_0} ^-]$ which has the following properties ${\rm (T)}_1$--${\rm (T)}_5$. 
\begin{enumerate}
\item[${\rm (T)}_1$] Each element in ${\bf B}_{w_0}$ is weakly pointed for all $t \in \mathbb{T}$. 
In particular, the extended $g$-vectors $g_{{\mathbf s}_t} (b)$, $b \in {\bf B}_{w_0}$, are defined.
\item[${\rm (T)}_2$] For every $t \in \mathbb{T}$, the map ${\bf B}_{w_0} \rightarrow \z^J$ given by $b \mapsto g_{{\mathbf s}_t} (b)$ is injective.
\item[${\rm (T)}_3$] If $t \overset{k}{\text{---}} t^\prime$ in $\mathbb{T}$, then it holds that $g_{{\mathbf s}_{t^\prime}} (b) = \mu_k ^T (g_{{\mathbf s}_t} (b))$ for all $b \in {\bf B}_{w_0}$, where $\mu_k^T$ is the tropicalized cluster mutation at $t$ given in \eqref{eq:tropicalized_cluster}.
\item[${\rm (T)}_4$] For each $\lambda \in P_+$, there exists a subset ${\bf B}_{w_0} [\lambda] \subseteq {\bf B}_{w_0}$ such that 
\[\iota_\lambda(H^0(G/B, \mathcal{L}_{\lambda})) = \sum_{b \in {\bf B}_{w_0} [\lambda]} \c b.\]
\item[${\rm (T)}_5$] For each $\lambda \in P_+$ and $v, w \in W$ such that $v \leq w$, there exists a subset ${\bf B}_{w_0} [\lambda, X_w^v] \subseteq {\bf B}_{w_0} [\lambda]$ such that $\{\Lambda_{X_w^v}^{(\lambda)} (b) \mid b \in {\bf B}_{w_0} [\lambda, X_w^v]\}$ forms a $\mathbb{C}$-basis of $H^0(X_w^v, \mathcal{L}_\lambda)$ and such that $\Lambda_{X_w^v}^{(\lambda)} (b) = 0$ for all $b \in {\bf B}_{w_0} [\lambda] \setminus {\bf B}_{w_0} [\lambda, X_w^v]$.
\end{enumerate}

It follows by property ${\rm (T)}_3$ that ${\bf B}_{w_0}$ is parametrized by tropical points for the Fock--Goncharov dual $\mathcal{A}^\vee$ (see \cite[Section 2 and Definition A.4]{GHKK}). 
The existence of such basis ${\bf B}_{w_0}$ was proved by Kashiwara--Kim \cite{KasKim} and Qin \cite{Qin2} as follows.

\begin{thm}[{see \cite{KasKim, Qin2}}]\label{t:existence_of_pointed}
There exists a $\c$-basis ${\bf B}_{w_0}$ of $\c[U_{w_0} ^-]$ having properties ${\rm (T)}_1$--${\rm (T)}_5$. 
\end{thm}

More precisely, the \emph{dual canonical basis} ($=$ the \emph{upper global basis}) of $\c[U_{w_0} ^-]$ has properties ${\rm (T)}_1$--${\rm (T)}_5$; we give more details in the following.
Since $U^- _{w_0}$ is given by $D_{w_0 \varpi_i, \varpi_i} \neq 0$ for $i \in I$ as an open subvariety of $U^-$, we have 
\[\c[U^-][D_{w_0 \varpi_i, \varpi_i}^{-1}\mid i\in I] \simeq \c[U^- _{w_0}]. \]
Through this isomorphism, the set 
\[{\bf B}_{w_0} ^{\rm up} \coloneqq \{G^{\rm up} (b) \cdot \prod_{i \in I} D_{w_0 \varpi_i, \varpi_i} ^{-a_i} \mid b \in \mathcal{B}(\infty),\ (a_i)_{i \in I} \in \z_{\geq 0} ^I\}\]
forms a $\c$-basis of $\c[U^- _{w_0}]$, which is called (the specialization at $q = 1$ of) the \emph{dual canonical basis}/\emph{upper global basis} of $\c[U_{w_0} ^-]$ (see \cite[Proposition 4.5 and Definition 4.6]{KimOya}).
By \cref{p:relation_between_crystals_G/B}, this basis ${\bf B}_{w_0} ^{\rm up}$ has property $({\rm T})_4$ for ${\bf B}_{w_0} [\lambda] \coloneqq \{G^{\rm up} (b) \mid b \in \widetilde{\mathcal{B}}(\lambda)\}$.
Then we see by the paragraph after \cref{p:properties of Demazure} that ${\bf B}_{w_0} ^{\rm up}$ also has property ${\rm (T)}_5$ for ${\bf B}_{w_0} [\lambda, X_w^v] \coloneqq \{G^{\rm up} (b) \mid b \in \widetilde{\mathcal{B}}_w^v(\lambda)\}$, where $\widetilde{\mathcal{B}}_w^v(\lambda)$ is defined as a subset of $\widetilde{\mathcal{B}}(\lambda)$ which corresponds to $\mathcal{B}_w(\lambda) \cap \mathcal{B}^v(\lambda) \subseteq \mathcal{B}(\lambda)$ under the bijective map $\pi_\lambda \colon \widetilde{\mathcal{B}}(\lambda) \rightarrow \mathcal{B}(\lambda)$ in \cref{p:relations_between_crystals}.
In addition, Kashiwara--Kim \cite[Lemmas 3.6, 3.12 and Theorem 3.16]{KasKim} and Qin \cite[Theorem 9.5.1]{Qin2} proved that properties $({\rm T})_1$--$({\rm T})_3$ hold for ${\bf B}_{w_0} ^{\rm up}$, which implies \cref{t:existence_of_pointed}.

\begin{rem}
In a quantum setting, the upper global basis is a common triangular basis in the sense of \cite[Definition 6.1.3]{Qin} by \cite[Proposition 3.19 and Remark 3.20]{KasKim} and by \cite[Theorem 9.5.1]{Qin2}.
\end{rem}

\begin{rem}
When the Cartan matrix of $\mathfrak{g}$ is symmetric, we see by \cite[Proposition 3.3]{KasKim} that the opposite dominance order $\preceq_{\varepsilon_t} ^{\rm op}$ in \cref{d:main_valuation} is the same as the partial order $\preceq_{\mathscr{S}}$ given in \cite[Section 3.3]{KasKim}.
In particular, by \cite[Lemma 3.6]{KasKim}, the extended $g$-vector $g_{\mathbf{s}_t}$ in \cref{d:pointed_elements} corresponds to ${\bf g}_{\mathscr{S}} ^L$ defined in \cite[Definition 3.8]{KasKim}.
\end{rem}

\subsection{Toric degenerations of flag varieties arising from cluster structures}

In this subsection, we review some results in \cite{FO2} which we use in the present paper.
For $\lambda \in P_+$ and $t \in \mathbb{T}$, write $\mathcal{S}_{\mathbf{s}_t} (\lambda) \coloneqq S(G/B, \mathcal{L}_\lambda, v_{\mathbf{s}_t}, \tau_\lambda)$, $\mathcal{C}_{\mathbf{s}_t} (\lambda) \coloneqq C(G/B, \mathcal{L}_\lambda, v_{\mathbf{s}_t}, \tau_\lambda)$, and $\Delta_{\mathbf{s}_t} (\lambda) \coloneqq \Delta(G/B, \mathcal{L}_\lambda, v_{\mathbf{s}_t}, \tau_\lambda)$.
We take a $\c$-basis ${\bf B}_{w_0}$ of $\c[U_{w_0} ^-]$ as in \cref{t:existence_of_pointed}.
By property ${\rm (T)}_1$ and \eqref{eq:g-vector_valuation}, we have $v_{\mathbf{s}_t} (b) = g_{\mathbf{s}_t} (b)$ for all $t \in \mathbb{T}$ and $b \in {\bf B}_{w_0}$. 

\begin{thm}[{see \cite[Theorem 7.1]{FO2} and its proof}]\label{t:NO_body_parametrized_by_seeds}
Let $\lambda \in P_+$, and $t \in \mathbb{T}$.
\begin{enumerate}
\item[{\rm (1)}] The Newton--Okounkov body $\Delta_{\mathbf{s}_t} (\lambda)$ does not depend on the choice of a refinement $\leq_t$ of the opposite dominance order $\preceq_{\varepsilon_t} ^{\rm op}$.
\item[{\rm (2)}] The Newton--Okounkov body $\Delta_{\mathbf{s}_t} (\lambda)$ is a rational convex polytope.
\item[{\rm (3)}] The equality $\Delta_{\mathbf{s}_t} (\lambda) \cap \z^J = \{g_{\mathbf{s}_t} (b) \mid b \in {\bf B}_{w_0} [\lambda]\}$ holds.
\item[{\rm (4)}] If $t \overset{k}{\text{---}} t^\prime$ in $\mathbb{T}$, then the equality $\Delta_{\mathbf{s}_{t^\prime}} (\lambda) = \mu_k ^T (\Delta_{\mathbf{s}_t} (\lambda))$ holds.
\item[{\rm (5)}] If $\lambda \in P_{++}$, then there exists a toric degeneration of $G/B$ to the irreducible normal projective toric variety $X(\Delta_{\mathbf{s}_t} (\lambda))$ corresponding to $\Delta_{\mathbf{s}_t} (\lambda)$. 
\end{enumerate}
\end{thm}

\begin{thm}[{see \cite[Theorem 6.5 and Corollary 6.6]{FO2} and \cite[Proposition 3.28, Theorem 5.1, and Corollary 5.3]{FO}}]\label{t:NO_body_parametrized_by_seeds_string}
Let ${\bm i} = (i_1, i_2, \ldots, i_N) \in R(w_0)$.
\begin{enumerate}
\item[{\rm (1)}] Define a unimodular transformation $\Upsilon_{\bm i} \colon \r^J \rightarrow \r^J$ by $\Upsilon_{\bm i} ({\bm a}) \coloneqq {\bm a} \cdot M_{\bm i}$, where $M_{\bm i} = (d_{k, \ell})_{k, \ell \in J}$ is a unimodular $J \times J$ matrix given by 
\[d_{k, \ell} \coloneqq \begin{cases}
\langle s_{i_{\ell+1}} \cdots s_{i_k} \varpi_{i_k}, h_{i_\ell} \rangle &{\text if}\ \ell \leq k,\\
0 &{\text if}\ \ell > k.
\end{cases}\]
Then it holds for $\lambda \in P_+$ that
\[\Delta_{\bm i} (\lambda) = \Upsilon_{\bm i}(\Delta_{\mathbf{s}_{\bm i}} (\lambda)).\]
\item[{\rm (2)}] For all $\lambda \in P_+$, the equality $\Phi_{\bm i} (\mathcal{B}(\lambda)) = \{\Upsilon_{\bm i} (g_{\mathbf{s}_{\bm i}} (b)) \mid b \in {\bf B}_{w_0} [\lambda]\}$ holds.
\end{enumerate}
\end{thm}

\begin{thm}[{see \cite[Theorem 6.24 and Corollary 6.25]{FO2} and \cite[Proposition 3.29, Theorem 5.1, and Corollary 5.3]{FO}}]\label{t:NO_body_parametrized_by_seeds_NZ}
Let ${\bm i} = (i_1, i_2, \ldots, i_N) \in R(w_0)$.
\begin{enumerate}
\item[{\rm (1)}] There exist an FZ-seed ${\mathbf s}_{\bm i} ^{\rm mut} \in \mathcal{S}$ and a unimodular $J \times J$ matrix $N_{\bm i}$ such that the equality
\[\widetilde{\Delta}_{\bm i} (\lambda) = \widetilde{\Upsilon}_{\bm i}(\Delta_{\mathbf{s}_{\bm i} ^{\rm mut}} (\lambda))\]
holds for all $\lambda \in P_+$, where $\widetilde{\Upsilon}_{\bm i} \colon \r^J \rightarrow \r^J$ is a unimodular transformation given by $\widetilde{\Upsilon}_{\bm i} ({\bm a}) \coloneqq {\bm a} \cdot N_{\bm i}$.
\item[{\rm (2)}] For all $\lambda \in P_+$, the equality $\Psi_{\bm i} (\mathcal{B}(\lambda)) = \{\widetilde{\Upsilon}_{\bm i} (g_{\mathbf{s}_{\bm i} ^{\rm mut}} (b)) \mid b \in {\bf B}_{w_0} [\lambda]\}$ holds.
\end{enumerate}
\end{thm}

\begin{rem}
Theorems \ref{t:NO_body_parametrized_by_seeds}, \ref{t:NO_body_parametrized_by_seeds_string}, \ref{t:NO_body_parametrized_by_seeds_NZ} are naturally extended to Schubert varieties $X_w$.
\end{rem}

\section{Main result}\label{s:main_result}

In this section, we prove our main result.
Let $\mathcal{S} = \{{\mathbf s}_t = (\mathbf{A}_t, \varepsilon_t)\}_{t \in \mathbb{T}}$ be a cluster pattern for the upper cluster algebra $\c[U_{w_0} ^-]$, and take a $\c$-basis ${\bf B}_{w_0}$ of $\c[U_{w_0} ^-]$ as in \cref{t:existence_of_pointed}.
For $t \in \mathbb{T}$, we set
\[\mathscr{S}_{\mathbf{s}_t} \coloneqq \bigcup_{\lambda \in P_+} \{(\lambda, g_{\mathbf{s}_t} (b)) \mid b \in {\bf B}_{w_0} [\lambda]\} \subseteq P_+ \times \z^J,\] 
and denote by $\mathscr{C}_{\mathbf{s}_t} \subseteq P_\r \times \r^J$ the smallest real closed cone containing $\mathscr{S}_{\mathbf{s}_t}$. 
In a way similar to the proof of \cite[Theorem 7.1 (2)]{FO2}, we obtain the following.

\begin{prop}
The real closed cone $\mathscr{C}_{\mathbf{s}_t}$ is a rational convex polyhedral cone, and it holds that $\mathscr{C}_{\mathbf{s}_t} \cap (P_+ \times \z^J) = \mathscr{S}_{\mathbf{s}_t}$.
\end{prop}

Define $\pi_1, \pi_2, \omega_\lambda$ as in Section \ref{ss:semi-toric}. 
Then we have
\begin{align*}
\mathscr{S}_{\mathbf{s}_t} \cap \pi_1 ^{-1} (\z_{>0} \lambda) = \omega_\lambda(\mathcal{S}_{\mathbf{s}_t} (\lambda))\quad {\rm and}\quad \mathscr{C}_{\mathbf{s}_t} \cap \pi_1 ^{-1} (\r_{\geq 0} \lambda) = \omega_\lambda(\mathcal{C}_{\mathbf{s}_t} (\lambda))
\end{align*}
for $\lambda \in P_+$.
In particular, it holds that $\pi_2(\mathscr{C}_{\mathbf{s}_t} \cap \pi_1 ^{-1} (\lambda)) = \Delta_{\mathbf{s}_t} (\lambda)$.
For $v, w \in W$ such that $v \leq w$, define a subset $\mathscr{S}_{\mathbf{s}_t}(X_w^v) \subseteq P_+ \times \mathbb{Z}^J$ by 
\[\mathscr{S}_{\mathbf{s}_t}(X_w^v) \coloneqq \bigcup_{\lambda \in P_+} \{(\lambda, g_{\mathbf{s}_t} (b)) \mid b \in {\bf B}_{w_0} [\lambda, X_w^v]\},\] 
and denote by $\mathscr{C}_{\mathbf{s}_t}(X_w^v) \subseteq P_\mathbb{R} \times \mathbb{R}^J$ the smallest real closed cone containing $\mathscr{S}_{\mathbf{s}_t}(X_w^v)$. 
For $\lambda \in P_+$, we write
\begin{align*}
&\Delta_{\mathbf{s}_t}(\lambda, X_w^v) \coloneqq \pi_2(\mathscr{C}_{\mathbf{s}_t}(X_w^v) \cap \pi_1 ^{-1} (\lambda)).
\end{align*}
The following is the main result of the present paper. 

\begin{thm}\label{t:main_result_semi-toric}
Let $t \in \mathbb{T}$, and $v, w \in W$ such that $v \leq w$.
\begin{enumerate}
\item[{\rm (1)}] The set $\mathscr{C}_{\mathbf{s}_t}(X_w^v)$ is a union of faces of $\mathscr{C}_{\mathbf{s}_t}$, and the following equality holds:
\begin{align*}
\mathscr{C}_{\mathbf{s}_t}(X_w^v) \cap (P_+ \times \z^J) &= \mathscr{S}_{\mathbf{s}_t}(X_w^v).
\end{align*}
\item[{\rm (2)}] For $\lambda \in P_+$, the set $\Delta_{\mathbf{s}_t}(\lambda, X_w^v)$ is a union of faces of $\Delta_{\mathbf{s}_t}(\lambda)$.
\item[{\rm (3)}] For $\lambda \in P_{++}$, the Richardson variety $X_w^v$ degenerates to the union of irreducible closed toric subvarieties of $X(\Delta_{\mathbf{s}_t}(\lambda))$ corresponding to the faces of $\Delta_{\mathbf{s}_t}(\lambda, X_w^v)$.
\end{enumerate}
\end{thm}

To prove \cref{t:main_result_semi-toric}, we use the following lemma in convex geometry as the proof of \cite[Proposition 3.5]{Mor}. 

\begin{lem}[{see, for instance, \cite[Lemma 3.6]{Mor}}]\label{l:union_of_faces}
Fix $m \in \z_{>0}$. 
Let $C \subseteq \r^m$ be a rational convex polyhedral cone, and assume that $C$ is strongly convex. 
If a subset $\Gamma \subseteq C \cap \z^m$ satisfies the following conditions:
\begin{enumerate}
\item[{\rm (i)}] ${\bm a} + {\bm a}^\prime \notin \Gamma$ for all ${\bm a} \in C \cap \z^m$ and ${\bm a}^\prime \in (C \cap \z^m) \setminus \Gamma$;
\item[{\rm (ii)}] $k {\bm a} \in \Gamma$ for all $k \in \z_{> 0}$ and ${\bm a} \in \Gamma$,
\end{enumerate}
then there uniquely exists a union $\Gamma_\r$ of faces of $C$ such that $\Gamma_\r \cap \z^m = \Gamma$. 
\end{lem}

We also need some properties of the extended $g$-vector $g_{\mathbf{s}_{\bm i}}$ for ${\bm i} \in R(w_0)$. 

\begin{prop}\label{p:g-vectors_for_Richardson}
Let ${\bm i} \in R(w_0)$, $\lambda \in P_+$, and $v, w \in W$ such that $v \leq w$.
Then it holds that
\begin{align*}
\Phi_{\bm i}\left(\mathcal{B}(\lambda) \setminus (\mathcal{B}_w(\lambda) \cap \mathcal{B}^v(\lambda))\right) &= \{\Upsilon_{\bm i} (g_{\mathbf{s}_{\bm i}} (b)) \mid b \in {\bf B}_{w_0} [\lambda] \setminus {\bf B}_{w_0} [\lambda, X_w^v]\},\\
\Phi_{\bm i}(\mathcal{B}_w(\lambda) \cap \mathcal{B}^v(\lambda)) &= \{\Upsilon_{\bm i} (g_{\mathbf{s}_{\bm i}} (b)) \mid b \in {\bf B}_{w_0} [\lambda, X_w^v]\}.
\end{align*}
\end{prop}

\begin{proof}
For ${\bm i} = (i_1, i_2, \ldots, i_N) \in R(w_0)$, we define a valuation $\tilde{v}_{\bm i} ^{\rm low} \colon \mathbb{C}(G/B) \setminus \{0\} \rightarrow \z^N$ as in \cite[Definition 2.10]{FO2}.
The definition of $\tilde{v}_{\bm i} ^{\rm low}$ is slightly different from the one in \cite[Section 2]{FO} because the order of coordinates of its values is reversed.
In the paper \cite{FO}, the author and Oya related the valuation $\tilde{v}_{\bm i} ^{\rm low}$ with the string parametrization $\Phi_{\bm i}$ using a specific $\c$-basis $\{\Xi^{\rm up}_\lambda(b) \mid b \in \mathcal{B}(\lambda)\}$ of $H^0(G/B, \mathcal{L}_\lambda)$ which has positivity properties. 
By \cite[Proposition 4.4]{FO}, the basis is compatible with Schubert varieties $X_w$, that is, $\{\pi_w (\Xi^{\rm up} _\lambda(b)) \mid b \in \mathcal{B}_w(\lambda)\}$ forms a $\mathbb{C}$-basis of $H^0(X_w, \mathcal{L}_\lambda)$ and $\pi_w (\Xi^{\rm up} _\lambda(b)) = 0$ for all $b \in \mathcal{B}(\lambda) \setminus \mathcal{B}_w(\lambda)$, where we recall that $\pi_w \colon H^0(G/B, \mathcal{L}_\lambda) \rightarrow H^0(X_w, \mathcal{L}_\lambda)$ denotes the restriction map.
In a way similar to the proof of \cite[Proposition 4.4]{FO}, the basis is compatible with opposite Schubert varieties $X^v$ and hence with Richardson varieties $X_w^v$. 
Since we have $\tilde{v}_{\bm i} ^{\rm low} (\iota_\lambda(\Xi^{\rm up}_\lambda(b))) = \Phi_{\bm i}(b)$ for $b \in \mathcal{B}(\lambda)$ by \cite[Proposition 3.28 and Corollaries 3.20 (2), 5.2]{FO}, it follows by \cref{p:property_valuation} that
\[\Phi_{\bm i}\left(\mathcal{B}(\lambda) \setminus (\mathcal{B}_w(\lambda) \cap \mathcal{B}^v(\lambda))\right) = \{\tilde{v}_{\bm i} ^{\rm low} (\iota_\lambda(\sigma)) \mid \sigma \in I_w^v(\lambda) \setminus \{0\}\},\]
where $I_w^v(\lambda)$ denotes the kernel of the restriction map $H^0(G/B, \mathcal{L}_\lambda) \rightarrow H^0(X_w^v, \mathcal{L}_\lambda)$. 
Then the equality $\tilde{v}_{\bm i} ^{\rm low} (\iota_\lambda(\sigma)) = \Upsilon_{\bm i}(v_{\mathbf{s}_{\bm i}} (\iota_\lambda(\sigma)))$ holds for $\sigma \in I_w^v(\lambda) \setminus \{0\}$ by \cite[Theorem 6.5]{FO2}. 
Since ${\bf B}_{w_0} [\lambda] \setminus {\bf B}_{w_0} [\lambda, X_w^v]$ is a $\c$-basis of $\iota_\lambda (I_w^v(\lambda))$ by property ${\rm (T)_5}$, the first equality of the proposition holds by property ${\rm (T)_2}$ and \cref{p:property_valuation}.
Now the second equality follows from the first equality since the unimodular transformation $\Upsilon_{\bm i}$ gives a bijective map between $\{g_{\mathbf{s}_{\bm i}} (b) \mid b \in {\bf B}_{w_0} [\lambda]\}$ and $\Phi_{\bm i} (\mathcal{B}(\lambda))$ by \cref{t:NO_body_parametrized_by_seeds_string} (2).
\end{proof}

\begin{rem}
Let ${\bm i} \in R(w_0)$. 
By \cref{p:g-vectors_for_Richardson}, we see that \cref{t:main_result_semi-toric} for the FZ-seed $\mathbf{s}_{\bm i}$ is precisely \cref{t:semi-toric_degenerations} up to $\Upsilon_{\bm i}$. 
Hence our semi-toric degenerations of Richardson varieties in \cref{t:main_result_semi-toric} generalize Morier-Genoud's semi-toric degenerations \cite{Mor}.  
\end{rem}

\begin{proof}[{Proof of \cref{t:main_result_semi-toric}}]
We first prove that $\mathscr{S}_{\mathbf{s}_t}(X_w^v)$ is the set of lattice points in a union of faces of $\mathscr{C}_{\mathbf{s}_t}$. 
To do that, we show conditions (i), (ii) in \cref{l:union_of_faces} for $\mathscr{S}_{\mathbf{s}_t}(X_w^v)$; note that the cone $\mathscr{C}_{\mathbf{s}_t}$ is strongly convex since it is included in $(\sum_{i \in I} \r_{\geq 0} \varpi_i) \times \r^N$ and since $\mathscr{C}_{\mathbf{s}_t} \cap (\{0\} \times \r^N) = \{(0, 0)\}$ (see \cite[Theorem 2.30]{KK2}).

To prove condition (i), let us take arbitrary elements ${\bm a} \in \mathscr{S}_{\mathbf{s}_t}$ and ${\bm a}^\prime \in \mathscr{S}_{\mathbf{s}_t} \setminus \mathscr{S}_{\mathbf{s}_t}(X_w^v)$.
Then there exist $\lambda, \lambda^\prime \in P_+$, $b \in {\bf B}_{w_0} [\lambda]$, and $b^\prime \in {\bf B}_{w_0} [\lambda^\prime] \setminus {\bf B}_{w_0} [\lambda^\prime, X_w^v]$ such that ${\bm a} = (\lambda, g_{\mathbf{s}_t} (b))$ and ${\bm a}^\prime = (\lambda^\prime, g_{\mathbf{s}_t} (b^\prime))$. 
Since $\tau_{\lambda+\lambda^\prime} = c \cdot \tau_{\lambda} \cdot \tau_{\lambda^\prime}$ for some $c \in \c^\times$, it follows that $b \cdot b^\prime \in \iota_{\lambda+\lambda^\prime} (H^0(G/B, \mathcal{L}_{\lambda+\lambda^\prime}))$.
Since $\Lambda_{X_w^v}^{(\lambda^\prime)} (b^\prime) = 0$ by property ${\rm (T)_5}$, we see that 
\begin{align*}
\Lambda_{X_w^v}^{(\lambda+\lambda^\prime)} (b \cdot b^\prime) &= (b \cdot b^\prime \cdot \tau_{\lambda+\lambda^\prime})|_{X_w^v}\quad(\text{by the definition of}\ \Lambda_{X_w^v}^{(\lambda+\lambda^\prime)})\\
&= c \cdot (b \cdot b^\prime \cdot \tau_{\lambda} \cdot \tau_{\lambda^\prime})|_{X_w^v}\\
&= c \cdot \Lambda_{X_w^v}^{(\lambda)} (b) \cdot \Lambda_{X_w^v}^{(\lambda^\prime)} (b^\prime) = 0.
\end{align*}
Hence property ${\rm (T)_5}$ implies that there exist $b_1, b_2, \ldots, b_\ell \in {\bf B}_{w_0} [\lambda+\lambda^\prime] \setminus {\bf B}_{w_0} [\lambda+\lambda^\prime, X_w^v]$ and $c_1, c_2, \ldots, c_\ell \in \c^\times$ such that 
\[b \cdot b^\prime = c_1 b_1 + c_2 b_2 + \cdots + c_\ell b_\ell.\]
Then we see by property ${\rm (T)_2}$ and \cref{p:property_valuation} that 
\[v_{\mathbf{s}_t} (b \cdot b^\prime) \in \{v_{\mathbf{s}_t} (b_1), \ldots, v_{\mathbf{s}_t} (b_\ell)\} = \{g_{\mathbf{s}_t} (b_1), \ldots, g_{\mathbf{s}_t} (b_\ell)\}.\]
Hence it holds that $(\lambda+\lambda^\prime, v_{\mathbf{s}_t} (b \cdot b^\prime)) \in \mathscr{S}_{\mathbf{s}_t} \setminus \mathscr{S}_{\mathbf{s}_t}(X_w^v)$. 
Since we have 
\begin{align*}
{\bm a} + {\bm a}^\prime &= (\lambda+\lambda^\prime, g_{\mathbf{s}_t} (b) + g_{\mathbf{s}_t} (b^\prime))\\
&= (\lambda+\lambda^\prime, v_{\mathbf{s}_t} (b) + v_{\mathbf{s}_t} (b^\prime)) = (\lambda+\lambda^\prime, v_{\mathbf{s}_t} (b \cdot b^\prime)),
\end{align*}
it follows that ${\bm a} + {\bm a}^\prime \notin \mathscr{S}_{\mathbf{s}_t}(X_w^v)$.
This proves condition (i). 

To prove condition (ii), we take arbitrary elements $k \in \z_{>0}$ and ${\bm a} \in \mathscr{S}_{\mathbf{s}_t}(X_w^v)$.
Then there exist $\lambda \in P_+$ and $b \in {\bf B}_{w_0} [\lambda, X_w^v]$ such that ${\bm a} = (\lambda, g_{\mathbf{s}_t} (b))$. 
Let us show 
\begin{equation}\label{eq:goal_extended_g-vector}
\begin{aligned}
k g_{\mathbf{s}_t} (b) \in \{g_{\mathbf{s}_t} (b^\prime) \mid b^\prime \in {\bf B}_{w_0} [k\lambda, X_w^v]\},
\end{aligned}
\end{equation}
which implies that $k {\bm a} = (k \lambda, k g_{\mathbf{s}_t} (b)) \in \mathscr{S}_{\mathbf{s}_t}(X_w^v)$.
Fix ${\bm i} \in R(w_0)$. 
Since $\mathbf{s}_t$ and $\mathbf{s}_{\bm i}$ are mutation equivalent, there exists a sequence $(\mu_{j_1}, \mu_{j_2}, \ldots, \mu_{j_p})$ of mutations such that $\mathbf{s}_{\bm i} = \mu_{j_p} \circ \cdots \circ \mu_{j_2} \circ \mu_{j_1} (\mathbf{s}_t)$. 
We write $\overline{\Upsilon} \coloneqq \Upsilon_{\bm i} \circ \mu_{j_p}^T \circ \cdots \circ \mu_{j_2}^T \circ \mu_{j_1}^T$.
Since $\Upsilon_{\bm i}$ and $\mu_{j_1}^T, \mu_{j_2}^T, \ldots, \mu_{j_p}^T$ are bijective and piecewise-linear, the composition $\overline{\Upsilon}$ is also a bijective piecewise-linear map. 
Hence, in order to prove \eqref{eq:goal_extended_g-vector}, it suffices to show that 
\[\overline{\Upsilon} (k g_{\mathbf{s}_t} (b))\ (= k \overline{\Upsilon} (g_{\mathbf{s}_t} (b))) \in \{\overline{\Upsilon}(g_{\mathbf{s}_t} (b^\prime)) \mid b^\prime \in {\bf B}_{w_0} [k\lambda, X_w^v]\}.\] 
By \cref{p:g-vectors_for_Richardson} and property ${\rm (T)_3}$, we have $\overline{\Upsilon}(g_{\mathbf{s}_t} (b)) = \Upsilon_{\bm i} (g_{\mathbf{s}_{\bm i}} (b)) \in \Phi_{\bm i}(\mathcal{B}_w(\lambda) \cap \mathcal{B}^v(\lambda))$ and 
\[\{\overline{\Upsilon}(g_{\mathbf{s}_t} (b^\prime)) \mid b^\prime \in {\bf B}_{w_0} [k\lambda, X_w^v]\} = \Phi_{\bm i}(\mathcal{B}_w(k\lambda) \cap \mathcal{B}^v(k\lambda)).\]
Hence the assertion follows by \cref{t:semi-toric_degenerations} (1).
This proves condition (ii), which concludes by \cref{l:union_of_faces} that $\mathscr{S}_{\mathbf{s}_t}(X_w^v)$ is the set of lattice points in a union of faces of $\mathscr{C}_{\mathbf{s}_t}$. 
This implies part (1) of the theorem by the definition of $\mathscr{C}_{\mathbf{s}_t}(X_w^v)$. 
Finally, part (2) of the theorem is an immediate consequence of part (1), and part (3) follows from parts (1), (2) in a way similar to the proof of \cite[Theorem 3.7]{Mor}.
\end{proof}

\begin{ex}
Let $G = SL_{n+1} (\c)$, $\lambda \in P_+$, and 
\[{\bm i}_A \coloneqq (1, 2, 1, 3, 2, 1, \ldots, n, n-1, \ldots, 1) \in \z^{\frac{n(n+1)}{2}}.\]
Then Littelmann \cite[Corollary 5]{Lit} proved that the string polytope $\Delta_{{\bm i}_A}(\lambda)$ is unimodularly equivalent to the Gelfand--Tsetlin polytope $GT(\lambda)$.
Under this unimodular transformation, the author \cite[Corollary 5.2]{Fuj2} proved that Morier-Genoud's semi-toric degeneration \cite{Mor} of $X^v$ for $\Delta_{{\bm i}_A}(\lambda)$ has the same limit as Kogan--Miller's semi-toric degeneration \cite{KoM} of $X^v$. 
In this sense, Morier-Genoud's semi-toric degenerations and hence our semi-toric degenerations can be regarded as generalizations of Kogan--Miller's semi-toric degenerations. 
\end{ex}

Let ${\bm i} \in R(w_0)$, $\lambda \in P_+$, and $v, w \in W$ such that $v \leq w$.
Replacing $\Phi_{\bm i}$ by $\Psi_{\bm i}$ in the definitions of $\mathscr{S}_{\bm i}, \mathscr{C}_{\bm i}, \mathscr{S}_{\bm i}(X_w^v), \mathscr{C}_{\bm i}(X_w^v)$, and $\Delta_{\bm i}(\lambda, X_w^v)$ in Section \ref{ss:semi-toric}, we define $\widetilde{\mathscr{S}}_{\bm i}, \widetilde{\mathscr{C}}_{\bm i}, \widetilde{\mathscr{S}}_{\bm i}(X_w^v), \widetilde{\mathscr{C}}_{\bm i}(X_w^v)$, and $\widetilde{\Delta}_{\bm i}(\lambda, X_w^v)$, respectively.

\begin{prop}[{see, for instance, \cite[Proposition 6.7]{Fuj}}] 
The real closed cone $\widetilde{\mathscr{C}}_{\bm i}$ is a rational convex polyhedral cone, and it holds that $\widetilde{\mathscr{C}}_{\bm i} \cap (P_+ \times \z^N) = \widetilde{\mathscr{S}}_{\bm i}$.
\end{prop}

In a way similar to the proof of \cref{p:g-vectors_for_Richardson}, we obtain the following.

\begin{prop}\label{p:g-vectors_for_Richardson_NZ}
Let ${\bm i} \in R(w_0)$, $\lambda \in P_+$, and $v, w \in W$ such that $v \leq w$.
Then it holds that
\begin{align*}
\Psi_{\bm i}\left(\mathcal{B}(\lambda) \setminus (\mathcal{B}_w(\lambda) \cap \mathcal{B}^v(\lambda))\right) &= \{\widetilde{\Upsilon}_{\bm i} (g_{\mathbf{s}_{\bm i}^{\rm mut}} (b)) \mid b \in {\bf B}_{w_0} [\lambda] \setminus {\bf B}_{w_0} [\lambda, X_w^v]\},\\
\Psi_{\bm i}(\mathcal{B}_w(\lambda) \cap \mathcal{B}^v(\lambda)) &= \{\widetilde{\Upsilon}_{\bm i} (g_{\mathbf{s}_{\bm i}^{\rm mut}} (b)) \mid b \in {\bf B}_{w_0} [\lambda, X_w^v]\}.
\end{align*}
\end{prop}

Combining \cref{t:main_result_semi-toric} with \cref{p:g-vectors_for_Richardson_NZ}, we deduce the following.

\begin{thm}
Let ${\bm i} \in R(w_0)$, and $v, w \in W$ such that $v \leq w$.
\begin{enumerate}
\item[{\rm (1)}] The set $\widetilde{\mathscr{C}}_{\bm i}(X_w^v)$ is a union of faces of $\widetilde{\mathscr{C}}_{\bm i}$, and the following equality holds:
\begin{align*}
\widetilde{\mathscr{C}}_{\bm i}(X_w^v) \cap (P_+ \times \z^N) &= \widetilde{\mathscr{S}}_{\bm i}(X_w^v).
\end{align*}
\item[{\rm (2)}] For $\lambda \in P_+$, the set $\widetilde{\Delta}_{\bm i}(\lambda, X_w^v)$ is a union of faces of $\widetilde{\Delta}_{\bm i}(\lambda)$.
\item[{\rm (3)}] For $\lambda \in P_{++}$, the Richardson variety $X_w^v$ degenerates to the union of irreducible closed toric subvarieties of $X(\widetilde{\Delta}_{\bm i}(\lambda))$ corresponding to the faces of $\widetilde{\Delta}_{\bm i}(\lambda, X_w^v)$.
\end{enumerate}
\end{thm}

\bibliographystyle{jplain} 
\def\cprime{$'$} 

\end{document}